\documentclass[sn-mathphys,Numbered]{sn-jnl}
\usepackage{amssymb,amsmath,amsfonts,amsthm,graphicx,verbatim,subcaption,tikz,pgfplots,float}

\usepackage[capitalise,nameinlink,noabbrev]{cleveref}
\crefname{equation}{}{}
\crefname{Assumption}{Assumption}{Assumptions}
\crefname{algocf}{Algorithm}{Algorithms}

\usepackage{graphicx,float,verbatim}
\usepackage{multirow}%

\usepackage{mathrsfs}%
\usepackage[title]{appendix}%
\usepackage{xcolor}%
\usepackage{textcomp}%
\usepackage{manyfoot}%
\usepackage{booktabs}%

\usepackage[framed,numbered,autolinebreaks,useliterate]{mcode}
\usepackage[algo2e,ruled]{algorithm2e}

\newcommand{\ie}{i.e., }
\newcommand{\st}{\mathrm{s.t.} }

\newcommand{\argmin}{\mathop{\arg\min}}

\newcommand{\Ttran}{\mathsf{T}}

\newcommand{\de}{\,\mathrm{d}}

\newcommand{\Rq}{\mathsf{Rq}}

\newcommand{\defi}{:=}
\newcommand{\normml}{\left\vert\kern-0.25ex\left\vert\kern-0.25ex\left\vert}
\newcommand{\normmr}{\right\vert\kern-0.25ex\right\vert\kern-0.25ex\right\vert}
\newcommand{\gap}{\mathsf{gap}}
\newcommand{\order}{\mathcal{O}}
\newcommand{\zero}{\mathbf{0}}

\newcommand{\grad}{\mathrm{grad}}
\newcommand{\Sm}{\mathcal{S}^{n\!-\!1}}

\newcommand{\R}{\mathbb{R}}

\newcommand{\figsize}{0.7\textwidth}

\crefname{lstlisting}{Code}{Codes}

\providecommand{\spa}[1]{\mathrm{span}\{#1\}}

\providecommand{\abs}[1]{\lvert#1\rvert}
\providecommand{\norm}[1]{\lVert#1\rVert}
\providecommand{\dual}[1]{\langle#1\rangle}

\providecommand{\bigabs}[1]{\bigl\lvert#1\bigr\rvert}
\providecommand{\Bigabs}[1]{\Bigl\lvert#1\Bigr\rvert}

\providecommand{\Bigdual}[1]{\Bigl\langle#1\Bigr\rangle}

\newtheorem{theorem}{Theorem}
\newtheorem{Assumption}[theorem]{Assumption}
\newtheorem{lemma}[theorem]{Lemma}

\newtheorem{corollary}[theorem]{Corollary}
\newtheorem{proposition}[theorem]{Proposition}%

\newtheorem{remark}[theorem]{Remark}%
\newtheorem{example}[theorem]{Example}
\newtheorem{definition}[theorem]{Definition}%

\begin{document}

\title[Riemannian Acceleration with Preconditioning]{Riemannian Acceleration with Preconditioning for symmetric eigenvalue problems}

\author*[1]{\fnm{Nian} \sur{Shao}}\email{nian.shao@epfl.ch}
\author[2]{\fnm{Wenbin} \sur{Chen}}\email{wbchen@fudan.edu.cn}

\affil*[1]{\orgdiv{Institute of Mathematics}, \orgname{EPFL}, \orgaddress{\city{Lausanne}, \postcode{1015},  \country{Switzerland}}}

\affil[2]{\orgdiv{School of Mathematical Sciences and Shanghai Key Laboratory for Contemporary Applied Mathematics}, \orgname{Fudan University}, \orgaddress{\city{Shanghai}, \postcode{200433}, \country{China}}}

\abstract{
	The analysis of the acceleration behavior of gradient-based eigensolvers with preconditioning presents a substantial theoretical challenge. In this work, we present a novel framework for preconditioning on Riemannian manifolds and introduce a metric, the leading angle, to evaluate preconditioners for symmetric eigenvalue problems. We extend the locally optimal Riemannian accelerated gradient method for Riemannian convex optimization to develop the Riemannian Acceleration with Preconditioning (RAP) method for symmetric eigenvalue problems, thereby providing theoretical evidence to support its acceleration.
Our analysis of the Schwarz preconditioner for elliptic eigenvalue problems demonstrates that RAP achieves a convergence rate of $1-C\kappa^{-1/2}$, which is an improvement over the preconditioned steepest descent method's rate of $1-C\kappa^{-1}$. The exponent in $\kappa^{-1/2}$ is sharp, and numerical experiments confirm our theoretical findings.}


\maketitle
\section{Introduction}
This paper is concerned with the smallest eigenvalue problem
\begin{equation*}
	Au_{1} = \lambda_{1}Mu_{1},
\end{equation*}
where $A$ and $M$ are both symmetric positive definite matrices, $\lambda_{1}$ is the smallest eigenvalue and $u_{1}$ is a corresponding $M$-unit eigenvector. Assume $\lambda_{1}$ is simple. 

\paragraph*{Large-scale symmetric eigenvalue problems}
The symmetric eigenvalue problem holds a prominent place in the realm of numerical analysis and scientific computing \cite{Bai2000,Parlett1998,Wilkinson1965,Saad2011}. 
Let
\begin{equation}
	\label{defRQ}
	\Rq(x)=\frac{x^{\Ttran}Ax}{x^{\Ttran}Mx},\quad \text{where}\quad x\neq \zero
\end{equation}
be the Rayleigh quotient. For large-scale problems, especially when the smallest eigenvalue and its corresponding eigenvector are of interest, a popular strategy \cite[Chapter~9.1.2]{Dyakonov1996} is that 
\begin{equation*}
	\lambda_{1} = \min_{x\neq \zero} \Rq(x)\quad \text{and}\quad  \Rq(u_{1})=\lambda_{1}.
\end{equation*}
In the scenario that the matrix is only available for matrix-vector multiplications, there are mainly two types of reliable and efficient numerical methods: Lanczos-type methods and gradient-type methods. 
Compared to Lanczos-type methods, gradient-type methods can utilize preconditioners efficiently and become highly effective for eigenvalue problems arising from the discretization of elliptic operators, as the rate of convergence can become independent of the mesh size when a good preconditioner is employed~\cite{Knyazev1998}.
In this paper, we mainly focus on gradient-type methods with preconditioning.
\paragraph*{Preconditioned eigensolvers}
Preconditioned eigensolvers have been studied for over four decades, with a rich history of research and development. For an in-depth exploration of the early developments in this field, we refer readers to D'yakonov's comprehensive book \cite{Dyakonov1996} and Knyazev's illuminating survey paper \cite{Knyazev1998}. 
Among the myriad of preconditioned eigensolvers, the Locally Optimal Block Preconditioned Conjugate Gradient (LOBPCG) method \cite{Knyazev2001} stands out as one of the most popular choices. However, despite its popularity, analyzing the convergence of LOBPCG remains a formidable challenge. Previous endeavors about the theoretical analysis of the LOBPCG method are mainly from two perspectives. One is the sharp estimation of the convergence rate for the Preconditioned Steepest Descent (PSD) method \cite{Knyazev2003,Neymeyr2011,Neymeyr2012,Argentati2017,Knyazev2009}. These work focuses on the analysis of the preconditioner while not addressing the previous
approximation, which is called momentum term. The other one is the understanding of the momentum term \cite{Ovtchinnikov2006,Ovtchinnikov2008}. Ovtchinnikov initiated this exploration by starting with the Jacobi orthogonal complement correction operator and provided an asymptotic convergence analysis. However, it is important to note that his work did not fully exploit the acceleration properties of LOBPCG, which has been observed by numerical results in \cite{Knyazev2001}.

\paragraph*{Convexity of symmetric eigenvalue problems}

The convexity of symmetric eigenvalue problems is a crucial issue for understanding the convergence of eigensolvers. In fact, a weakened form of this convexity, known as the Polyak--Lojasiewicz condition, is a fundamental result in some textbook, such as \cite[Theorem~5.5]{Demmel1997}. Although it is not as strong as strong convexity, it is still sufficient to guarantee linear convergence of gradient-type methods.

To establish a strong-convexity-like structure in Euclidean space, a common technique involves removing the homogeneous direction of the Rayleigh quotient. For instance, in \cite{Shao2023a}, Shao, Chen and Bai proposed an implicit convexity structure by minimizing the Rayleigh quotient in the tangent plane of an eigenvector approximation. This approach led to the development of a provable accelerated eigensolver based on Preconditioning and Implicit Convexity (EPIC). Another example of creating a convexity structure in Euclidean space is found in \cite{Li2019}, where a low-rank approximation problem is considered as a replacement for Rayleigh quotient optimization.

In contrast, a more natural approach to developing a convexity structure for eigenvalue problems involves working on Riemannian manifolds as 
\begin{equation}
	\label{rmoptWOP}
	\begin{aligned}
		\min\quad & x^{\Ttran}Ax    \\
		\st\quad  & x^{\Ttran}Mx=1.
	\end{aligned}
\end{equation}
There are several results concerning the geodesic convexity of the symmetric eigenvalue problem, arising from its interpretation as an optimization problem with orthogonality constraints \cite{Absil2009,Edelman1998,Jiang2015,Dai2020}. One such example is a systematic study of geodesic convexity for symmetric eigenvalue problems by Alimisis and Vandereycken in \cite{Alimisis2022}.

\paragraph*{Acceleration on Riemannian optimization}
The gradient method, with its roots tracing back to Euler and Lagrange, stands as one of the earliest approaches for convex optimization problems. Forty years ago, Nesterov \cite{Nesterov1983} proposed an Accelerated Gradient (NAG) method, which has attracted significant practical and theoretical interest, as seen in works such as \cite{Beck2009, Becker2011, Nesterov2005, Odonoghue2015}. In 2014, Su, Boyd, and Candes proposed a dynamical system analogy for the NAG method in \cite{Su2014}. This novel perspective catalyzed extensive exploration of the NAG method over the past decade \cite{Shi2019, Shi2021, Muehlebach2019, Muehlebach2021, Luo2021, Park2021}. 

Turning to Riemannian optimization \cite{Absil2009, Boumal2023}, the body of work in this domain is comparatively less extensive. The earliest research involving accelerated methods on Riemannian manifolds can be traced back to \cite{Liu2017}. Subsequently, Zhang and Sra presented the first computationally tractable Riemannian acceleration method in \cite{Zhang2018}, analyzing its convergence by extending the estimate sequence technique from traditional NAG flow analysis. Recently, several Riemannian accelerated gradient methods have been proposed for Riemannian optimization problems, including \cite{Ahn2020, Alimisis2020, Duruisseaux2022, Kim2022, Alimisis2021a, MartinezRubio2022}.

\paragraph*{Paper organization}
In Section~2, we explore the geodesic convexity of preconditioned symmetric eigenvalue problems. Beginning with a continuous analogy, we introduce a new formulation for incorporating preconditioning on Riemannian manifolds, which differs from the classical approach presented by Mishra and Sepulchre in \cite{Mishra2016}. Unlike the analysis of eigensystems without preconditioning, preconditioned eigenvalue problems present unique theoretical challenges. To address these challenges, we develop new metrics for evaluating preconditioners, such as the leading angle. Finally, we provide estimates of local geodesic convexity and Lipschitz smoothness for symmetric eigenvalue problems under preconditioning.

In Section~3, we propose a Locally Optimal Riemannian Accelerated Gradient (LORAG) method for general Riemannian convex optimization problems, assuming only the local geodesic convexity and Lipschitz smoothness, and prove the acceleration.

Drawing from the theoretical foundations laid in Section~2 and the algorithmic insights from Section~3, we propose the Riemannian Acceleration with Preconditioning (RAP) method for preconditioned eigensystems in Section~4. Practical implementation aspects are also discussed. For a theoretical viewpoint, when the initial vector $x_{0}$ is sufficiently close to an eigenvector $u_{1}$, the rate of convergence can be expressed as:
\begin{equation}
	\label{introrate}
	\Rq(x_{m})-\lambda_{1}\leq 2\Bigl(1-\frac{1}{2\sqrt{\kappa_{B}}}\Bigr)^{m}\bigl(\Rq(x_{0})-\lambda_{1}\bigr),
\end{equation}
where $\Rq(x)$ is the Rayleigh quotient, $\lambda_{1}$ is the smallest eigenvalue and
\begin{equation*}
	\kappa_{B} \approx \frac{\nu_{\max}}{\nu_{\min}} \frac{1-\lambda_{1}/\lambda_{n}}{1-\lambda_{1}/\lambda_{2}},
\end{equation*}
with $\nu_{\min}$ and $\nu_{\max}$ representing the smallest and largest eigenvalues of $B^{-1}A$, respectively. Here, $B$ serves as a preconditioner for $A$, while $\lambda_{2}$ and $\lambda_{n}$ are the second-smallest and the largest eigenvalues of $M^{-1}A$, respectively. The notation ``$\approx$'' indicates that higher-order terms are ignored, as will be discussed later in \cref{rmkPEVP}. 
Compared with the sharp estimation of PSD method \cite{Argentati2017} as 
\begin{equation*}
	\frac{\Rq(x_{m+1})-\lambda_{1}}{\lambda_{2}-\Rq(x_{m+1})}\leq \Bigl(1-\frac{2(1-\lambda_{1}/\lambda_{2})}{\nu_{\max}/\nu_{\min}+1}\Bigr)^{2}\frac{\Rq(x_{m})-\lambda_{1}}{\lambda_{2}-\Rq(x_{m})},
\end{equation*}
RAP improves the exponent of the dominant term $\frac{1-\lambda_{1}/\lambda_{2}}{\nu_{\max}/\nu_{\min}+1}$, which is close to $0$ for challenging problems, from $1$ to $1/2$ when ignoring higher-order terms. Moreover, the convergence rate \cref{introrate} is essentially same as the conjecture for LOPCG in \cite[(5.5)]{Knyazev2001}, whose sharpness of exponent is verified in numerical experiments in \cite{Knyazev2001}.

To facilitate a deeper comprehension of our theoretical findings, in Section~5, we explore a specific example, \ie solving elliptic eigenvalue problems with Schwarz preconditioners.
When the overlapping or non-overlapping domain decomposition method is applied, with sufficiently small coarse mesh size $H$, the rate of convergence is established as
\begin{equation*}
	\Rq(x_{m})-\lambda_{1}\leq 2\biggl(1-C_{\text{RAP}}\Bigl(\frac{1-\lambda_{1}/\lambda_{2}}{\nu_{\max}/\nu_{\min}}\Bigr)^{1/2}\biggr)^{m}\bigl(\Rq(x_{0})-\lambda_{1}\bigr),
\end{equation*}
where $C_{\text{RAP}}$ is a constant independent of the coarse and fine mesh sizes and the eigenvalues gaps. Compared with the rate of convergence for the PSD method as
\begin{equation*}
	\Rq(x_{m})-\lambda_{1}\leq \Bigl(1-C_{\text{PSD}}\frac{1-\lambda_{1}/\lambda_{2}}{\nu_{\max}/\nu_{\min}}\Bigr)\bigl(\Rq(x_{m-1})-\lambda_{1}\bigr),
\end{equation*}
RAP demonstrates a clear acceleration advantage in the exponent of $\frac{1-\lambda_{1}/\lambda_{2}}{\nu_{\max}/\nu_{\min}}$.

Numerical experiments are presented in Section~6, offering illustrations of the effectiveness and efficiency of acceleration and preconditioning techniques.

For the self-containess, we briefly review some concepts in Riemannian manifolds related to our analysis, and give some examples on sphere in \cref{appRO}. For more details, we refer to the literatures about Riemannian manifolds in \cite{Lee2018,Klingenberg1995} and about Riemannian optimization in \cite{Boumal2023,Absil2009}. Moreover, we defer some technical proofs to \cref{appTS} for better readability.
\paragraph*{Notations}
For a symmetric positive definite matrix $A$, the notations $\dual{\cdot,\cdot}_{A}$ and $\norm{\cdot}_{A}$ represent the corresponding inner-product and norm, respectively. For standard Euclidean inner product and norm, the subscripts are omitted. We use $\Sm\defi \{x\in\R^{n}\mid x^{\Ttran}x=1\}$ to denote the standard sphere and $T_{x}\Sm\defi \{v\in\R^{n}\mid v^{\Ttran}x=0,\, x\in\Sm\}$ to denote the tangent space of $x$.

\section{Geodesic convexity of preconditioned eigensystems}
	Given a symmetric positive definite preconditioner $B$, the PSD method \cite{Samokish1958,Neymeyr2012} is
\begin{equation}
	\label{psd}
	x_{k+1} = x_{k}-\alpha_{k}B^{-1}\nabla \Rq(x_{k}),
\end{equation}
where $\Rq$ is the Rayleigh quotient defined in \cref{defRQ} and $\alpha_{k}>0$ is a step-size. In this section, we study the PSD method on Riemannian manifolds.

\subsection{Incorporate preconditioning on Riemannian manifolds}
Consider the following Riemannian optimization problem:
\begin{equation}
	\label{rmopt}
	\begin{aligned}
		\min\quad & \Rq(x)            \\
		\st\quad  & x^{\Ttran}Bx=1.
	\end{aligned}
\end{equation}
The following lemma forms a connection between the continuous analogy of the PSD method \cref{psd} and the Riemannian optimization problem \cref{rmopt}.
\begin{lemma}
	\label{conservationLaw}
	Consider the following preconditioned gradient flow:
	\begin{equation}
		\label{conflow}
		\left\{
		\begin{aligned}
			\frac{\de x}{\de t} & = -B^{-1}\nabla \Rq(x), \\
			x(0)                & =x_{0}.
		\end{aligned}
		\right.
	\end{equation}
	The gradient flow $x(t)$ from \cref{conflow} lies on a $B$-sphere for all $t>0$, \ie $\norm{x(t)}_{B}=\norm{x_{0}}_{B}$. 
\end{lemma}
\begin{proof}
	From the definition of the Rayleigh quotient $\Rq(x)$, we know
	\begin{equation}
		\label{gradf}
		\nabla \Rq(x) = \frac{2}{x^{\Ttran}Mx}\bigl(Ax-\Rq(x)Mx\bigr).
	\end{equation}
	Then a conservation law of $\norm{x(t)}_{B}$ is obtained as 
	\begin{equation*}
		\frac{\de \norm{x}_{B}^{2}}{\de t} = 2\Bigdual{\frac{\de x}{\de t},x}_{B} = -2\dual{\nabla \Rq(x),x}=0,
	\end{equation*}
	where the orthogonality of $\nabla \Rq(x)$ and $x$ as $\dual{\nabla \Rq(x),x}=0$ is used.
\end{proof}

The concept discussed in \cref{conservationLaw} is relatively straightforward and  illustrates how the choice of a preconditioner affects gradient flow. Due to the homogeneity property of the Rayleigh quotient, $\Rq(x) = \Rq(\alpha x)$ for any nonzero scalar $\alpha$, the minimum value of $\Rq(x)$ on any $B$-sphere is always $\lambda_{1}$. However, selecting an appropriate preconditioner $B$ can speed up the convergence process, as demonstrated in \cite{Knyazev1998}. 

As a numerical example in \cref{sec:numexp} shows, using an effective preconditioner, such as a two-level overlapping domain decomposition preconditioner, can significantly reduce the number of iterations required to solve Laplacian eigenvalue problems, bringing it down to a near-constant level.
This can be interpreted as the preconditioner guiding the trajectory to converge more quickly.

This observation also highlights a link between preconditioning and Riemannian optimization, particularly through the imposition of constraints that ensure the iterates remain on a $B$-sphere. The preconditioned eigensystem is mathematically represented in \cref{rmopt}. Compared to the formulation in \cref{rmoptWOP}, \cref{rmopt} appears slightly more expensive, as the denominator of Rayleigh quotient is not eliminated by the manifold constraints. However, this trade-off results in faster convergence by operating within a more favorable sphere.

The remainder of this section provides a mathematical explanation for this phenomenon by examining the condition number $\kappa = L/\mu$, which is the ratio of the Lipschitz smoothness parameter $L$ to the geodesic convexity parameter $\mu$ of the Rayleigh quotient on different spheres\footnote{See \cref{defGCF,defLS} in \cref{appRO}.}. As shown in \cite[Theorem 11.29]{Boumal2023}, the condition number significantly influences the convergence rate of the Riemannian gradient descent method, which is $1 - \kappa^{-1}$.

\subsection{Analysis on $M$-sphere: without preconditioning}

Let us first examine eigensystems without preconditioning as in \cref{rmoptWOP}. The geodesic convexity and Lipschitz smoothness of the Riemannian optimization problem \cref{rmoptWOP} have been established in previous works \cite{Alimisis2022}. However, our analysis differs slightly, as we define the convex region through the Rayleigh quotient in \cref{lemcvxset}.

For simplicity of our discussion, we first use some transformation to transform $M$-sphere to a standard sphere. For any $z^{\Ttran}Mz=1$, let $x=M^{1/2}z$. Then we know that
\begin{equation*}
	x^{\Ttran}x=z^{\Ttran}Mz=1\quad\text{and}\quad x^{\Ttran}A_{M}x=z^{\Ttran}Az,
\end{equation*}
where $A_{M}=M^{-1/2}AM^{-1/2}$. In this part, we consider
\begin{equation}
	\label{stdopt}
	\min_{x^{\Ttran}x=1}\quad \Rq_{M}(x),\quad \text{where}\quad \Rq_{M}(x)\defi x^{\Ttran}A_{M}x.
\end{equation}
A direct result from the transformation is that $\{\lambda_{i}\}_{i=1}^{n}$ are eigenvalues of $A_{M}$, and $u_{M}=\frac{M^{1/2}u_{1}}{\norm{u_{1}}_{M}}$ is a unit eigenvector corresponding to $\lambda_{1}$.
Now let us define a geodesically convex set\footnote{See \cref{defGCS} in \cref{appRO}.} $\mathcal{X}_{M}$, a neighborhood of $u_{M}$ measured by $\Rq_{M}$, for computation. 
\begin{lemma}
	\label{lemcvxset}
	Suppose $\lambda_{1}\leq\rho< (\lambda_{1}+\lambda_{2})/2$, then the set
	\begin{equation*}
		\mathcal{X}_{M} = \{x\in\Sm,\,x^{\Ttran}u_{M}>0\mid x^{\Ttran}A_{M}x\leq\rho\}
	\end{equation*}
	is geodesically convex.
\end{lemma}
\begin{proof}
	See \cref{applemcvxset}.
\end{proof}

Similar to \cite[Corollary~19]{Alimisis2022}, we can establish the geodesic convexity and Lipschitz smoothness of the optimization problem \cref{stdopt}.
\begin{theorem}
	\label{lemcvxfun}
	Suppose $\lambda_{1}\leq\rho<(\lambda_{1}+\lambda_{2})/2$, let $\Rq_{M}(x)=x^{\Ttran}A_{M}x$, then for any $x\in\mathcal{X}_{M}$ and $v\in T_{x}\Sm$, the following inequality holds:
	\begin{equation*}
		2(\lambda_{1}+\lambda_{2}-2\rho)\norm{v}^{2} \leq
		\dual{v,\nabla^{2}\Rq_{M}(x)[v]}\leq
		2(\lambda_{n}-\lambda_{1})\norm{v}^{2}.
	\end{equation*}
	These inequalities imply $\Rq_{M}$ is $2(\lambda_{1}+\lambda_{2}-2\rho)$-geodesically convex and $2(\lambda_{n}-\lambda_{1})$-Lipschitz smooth.
\end{theorem}
\begin{proof}
	Without loss of generality, assume $\norm{v}=1$, we know\footnote{Details are provided in \cref{appHess} in \cref{appRO}.}
	\begin{equation*}
		\dual{v,\nabla^{2}\Rq_{M}(x)[v]} = 2v^{\Ttran}\bigl(A_{M}-\Rq_{M}(x)I\bigr)v=2\Rq_{M}(v)-2\Rq_{M}(x).
	\end{equation*}
	With the following bound of $\Rq_{M}(v)$ in \cite[Lemma~3.1]{Notay2002}:
	\begin{equation}
		\label{estRqMv}
		\lambda_{1}+\lambda_{2}-\Rq_{M}(x)\leq \Rq_{M}(v)\leq \lambda_{n},
	\end{equation}
	and $\lambda_{1}\leq\Rq_{M}(x)\leq\rho<(\lambda_{1}+\lambda_{2})/2$,	we prove the desired result.
\end{proof}
\begin{corollary}
	\label{rmkEVP}
	The first-order approximation for the condition number of \cref{stdopt} is
	\begin{equation*}
		\kappa=\frac{2(\lambda_{n}-\lambda_{1})}{2(\lambda_{1}+\lambda_{2}-2\rho)}=\frac{\lambda_{n}}{\lambda_{2}}\frac{1-\lambda_{1}/\lambda_{n}}{1-\lambda_{1}/\lambda_{2}}\Bigl(1+\order\bigl(\frac{\rho-\lambda_{1}}{\lambda_{2}-\lambda_{1}}\bigr)\Bigr).
	\end{equation*}
	For elliptic eigenvalue problems, the condition number tends to hinder the performance of first-order methods, primarily due to the small relative spectral gap $\frac{\lambda_{2}-\lambda_{1}}{\lambda_{n}-\lambda_{1}}$.
\end{corollary}

\subsection{Challenge in analysis on $B$-sphere}
In order to control the condition number discussed in \cref{rmkEVP}, we need to study the preconditioned eigensystem \cref{rmopt}. Let
\begin{equation}
	\label{transformB}
	A_{B} = B^{-1/2}AB^{-1/2}\quad\text{and}\quad M_{B}=B^{-1/2}MB^{-1/2},
\end{equation}
the optimization problem \cref{rmopt} becomes
\begin{equation}
	\label{preconOpt}
	\min_{x\in\Sm}\Rq_{B}(x) ,\quad \text{where}\quad\Rq_{B}(x)\defi \frac{x^{\Ttran}A_{B}x}{x^{\Ttran}M_{B}x},
\end{equation}
and the eigenvalues of $(A_{B},M_{B})$ are $\{\lambda_{i}\}_{i=1}^{n}$, and $u_{B}=\frac{B^{1/2}u_{1}}{\norm{u_{1}}_{B}}$ is a unit eigenvector corresponding to $\lambda_{1}$.
Similar to \cref{lemcvxset}, which defines a geodesically convex set $\mathcal{X}_{M}$, another geodesically convex set $\mathcal{X}_{B}$ needs to be defined.
\begin{lemma}
	\label{lemcvxsetP}
	Suppose $\lambda_{1}\leq\rho< (\lambda_{1}+\lambda_{2})/2$, then the set
	\begin{equation*}
		\mathcal{X}_{B} = \{x\in\Sm,\, x^{\Ttran}u_{B}>0\mid \frac{x^{\Ttran}A_{B}x}{x^{\Ttran}M_{B}x}\leq\rho\}
	\end{equation*}
	is geodesically convex.
\end{lemma}
\begin{proof}
	See \cref{applemcvxsetP}
\end{proof}

Mimicking \cref{lemcvxfun}, we now discuss the Riemannian Hessian of $\Rq_{B}$ in $\mathcal{X}_{B}$. Unlike the previous case, preconditioning complicates the formulation of the Hessian. For any $x\in\mathcal{X}_{B}$ and $v\in T_{x}\Sm$ with $\norm{v}=1$, the Hessian can be computed directly as follows\footnote{Details are provided in \cref{appHess} in \cref{appRO}.}:
\begin{equation}
	\label{Hess}
	\dual{v,\nabla^{2}\Rq_{B}(x)[v]} = J_{1}-J_{2},
\end{equation}
where
\begin{equation*}
	J_{1}  = \frac{2\norm{v}_{A_{B}}^{2}\Rq_{B}(x)}{\norm{x}_{A_{B}}^{2}}\Bigl(1-\frac{\Rq_{B}(x)}{\Rq_{B}(v)}\Bigr)\text{ and }
	J_{2}  = \frac{8(v^{\Ttran}M_{B}x)}{\norm{x}_{M_{B}}^{4}}\Bigl(v^{\Ttran}\bigl(A_{B}x-\Rq_{B}(x)M_{B}x\bigr)\Bigr).
\end{equation*}
Even for the case $x=u_{B}$, the desired eigenvector of $(A_{B},M_{B})$, the cross term becomes $J_{2}=\zero$, but a lower bound of
\begin{equation*}
	\dual{v,\nabla^{2}\Rq_{B}(u_{B})[v]} = \frac{2\lambda_{1}\norm{v}_{A_{B}}^{2}}{\norm{u_{B}}_{A_{B}}^{2}}\Bigl(1-\frac{\lambda_{1}}{\Rq_{B}(v)}\Bigr)
\end{equation*}
is hard to analyze as it is difficult to determine a lower bound for $\Rq_{B}(v)$ in a manner analogous to that described in \cref{estRqMv}. The primary challenge arises from the fact that $v$ is orthogonal to $u_{B}$ in standard inner-product rather than in the $A_{B}$ (or $M_{B}$) inner-product. In examining the original problem \cref{rmopt} prior to the transformation \cref{transformB}, the crux of the issue lies in the dearth of information regarding the behavior of $\Rq(v)$ for $v$ in the $B$-orthogonal complement of $u_{1}$.

In the ideal, albeit unrealistic, scenario where the preconditioner is set to $B=A$, it is evident that $\Rq(v)\geq \lambda_{2}$ as $v$ is $A$-orthogonal to the eigenvector $u_{1}$. In light of the aforementioned considerations, a natural question arises with regard to a general preconditioner $B$. The question thus arises as to the extent of the $A$-orthogonality of $u_{1}$ and the $B$-orthogonal complement of $u_{1}$. Furthermore, if we consider $x$ in a neighborhood of $u_{1}$, can any $A$-orthogonal conditions between $x$ and the $B$-orthogonal complement of $x$ still be identified?
In order to address these questions, a new measure for preconditioned eigensystems, called leading angle, is introduced.

\subsection{Leading angle}

\begin{definition}[Leading Angle]
	\label{defvartheta}
	Suppose $(A,M)$ is a symmetric positive definite matrix pencil with eigenvalues $0<\lambda_{1}<\lambda_{2}\leq\dotsb\leq \lambda_{n}$, and $B$ is a symmetric positive definite matrix. Let $\Rq(x)$ be the Rayleigh quotient defined in \cref{defRQ} and $\lambda_{1}\leq\rho< \lambda_{2}$ be a scalar. Define the leading angle $\vartheta$ as
	\begin{equation*} 	
		\vartheta(B,\rho;\Rq) = \inf_{\Rq(x)\leq\rho}
                \inf_{v^{\Ttran}Bx=0} 
                \arccos\Bigl(\frac{\abs{v^{\Ttran}Ax}}{\norm{v}_{A}\norm{x}_{A}}\Bigr).
	\end{equation*}
	We denote $\vartheta = \vartheta(B,\rho;\Rq)$
	for the simplicity of notations.
\end{definition}

The leading angle $\vartheta$ directly addresses the question posed above. As shown in \cref{figLA}, it quantifies the $A$-angle between $x$ and $v$, where $x$ is near the desired eigenvector $u_{1}$ and $v$ lies in the $B$-orthogonal complement of $x$. A direct result is that 
\begin{equation*}
	\frac{\abs{u_{1}^{\Ttran}Av}}{\norm{u_{1}}_{A}\norm{v}_{A}}\leq \cos\vartheta\quad \text{for any } v^{\Ttran}Bu_{1}=0.
\end{equation*}

\begin{figure}[htbp]
	\centering
	\includegraphics[width=\figsize]{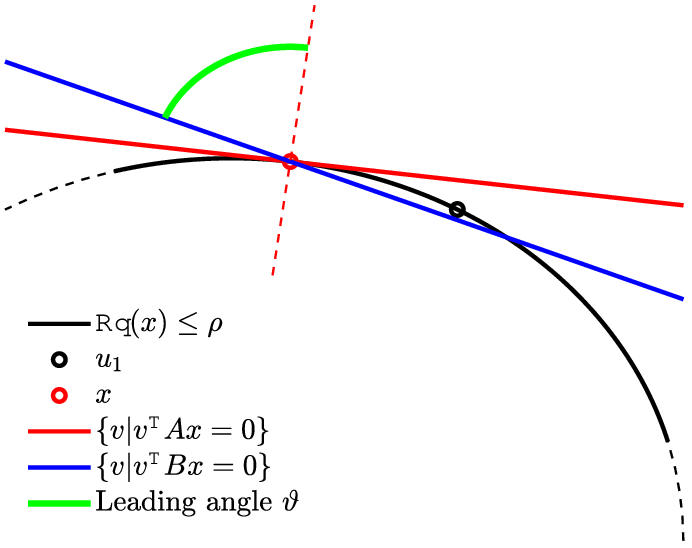}
	\caption{Leading angle $\vartheta$. The figure is under $A$ inner-product.}
	\label{figLA}
\end{figure}

Ideally, $\vartheta$ is close to $\pi/2$, or mathematical equivalently, $\cos\vartheta$ should approach zero. The following lemma establishes an equivalence between $\vartheta=\pi/2$ and $B=\mu A$ for some $\mu>0$.
\begin{lemma}
	\label{eqvAB}
	With notations in \cref{defvartheta}, suppose $\lambda_{1}<\rho<\lambda_{2}$ then $\vartheta=\pi/2$ if and only if $B$ is a multiple of $A$, \ie $A=\mu B$ for some $\mu>0$.
\end{lemma}
\begin{proof}
	See \cref{appeqvAB}.
\end{proof}
Besides the leading angle defined in \cref{defvartheta}, a traditional measurement for the quality of preconditioners is $\kappa_{\nu} = \nu_{\max}/\nu_{\min}$, where
\begin{equation}
	\label{defnu}
	\begin{aligned}
		\nu_{\min} &=\lambda_{\min}(B^{-1}A)=\min_{x\in\R^{n}}\frac{x^{\Ttran}Ax}{x^{\Ttran}Bx} = \min_{x\in\R^{n}}\frac{x^{\Ttran}AB^{-1}Ax}{x^{\Ttran}Ax}=\min_{x\in\R^{n}}\frac{x^{\Ttran}Bx}{x^{\Ttran}BA^{-1}Bx},\\ 
		\nu_{\max} &=\lambda_{\max}(B^{-1}A)=\max_{x\in\R^{n}}\frac{x^{\Ttran}Ax}{x^{\Ttran}Bx} = \max_{x\in\R^{n}}\frac{x^{\Ttran}AB^{-1}Ax}{x^{\Ttran}Ax}=\max_{x\in\R^{n}}\frac{x^{\Ttran}Bx}{x^{\Ttran}BA^{-1}Bx}.
	\end{aligned}
\end{equation}

In comparison to $\kappa_{\nu}$, the approximation properties of the neighborhood of $u_{1}$ are of greater concern when considering the the leading angle $\vartheta$. It is crucial to acknowledge that the prerequisites for preconditioners in solving linear systems diverge from those in eigenvalue problems. In the context of linear systems of the form $Ax = b$, an effective preconditioner must demonstrate robust performance across a range of right-hand sides $b$. This necessitates a condition number $\kappa_{\nu}$, defined as the condition number of $B^{-1}A$, which is typically small to ensure overall efficiency. 
In contrast, when dealing with the smallest eigenvalue problem, $Au_{1} = \lambda_{1} Mu_{1}$, the right-hand side is confined to the eigenspace spanned by $Mu_1$. Therefore, it is crucial to assess the behavior of $B^{-1}A$ specifically within a neighborhood of this eigenspace.

While the traditional measurement $\kappa_{\nu}$ does not accurately reflect the quality of $B$ within a neighborhood of the eigenspace, it still provides a global bound for $\vartheta$.
\begin{lemma}
	\label{varthetabdd}
	With notations in \cref{defvartheta}, it holds that 
	\begin{equation*}
		\cos\vartheta \leq \sqrt{1-\kappa_{\nu}^{-1}}.
	\end{equation*}
\end{lemma}
\begin{proof}
	For any $v^{\Ttran}Bx=0$, by the Cauchy--Schwarz inequality, we have
	\begin{equation*}
		\begin{aligned}
			\frac{\abs{x^{\Ttran}Av}}{\norm{x}_{A}\norm{v}_{A}} &=\min_{\alpha\in\R}\frac{\bigabs{\alpha x^{\Ttran}Bv+x^{\Ttran}(A-\alpha B)v}}{\norm{x}_{A}\norm{v}_{A}}=\min_{\alpha\in\R}\frac{\bigabs{x^{\Ttran}(A-\alpha B)v}}{\norm{x}_{A}\norm{v}_{A}}   \\
			&\leq \min_{\alpha\in\R} \frac{\norm{A^{-1}(A-\alpha B)x}_{A}\norm{v}_{A}}{\norm{x}_{A}\norm{v}_{A}}
			=\min_{\alpha\in\R} \frac{\norm{(I-\alpha A^{-1}B)x}_{A}}{\norm{x}_{A}}\\ 
			&= \frac{\Bigl(x^{\Ttran}Ax-\frac{(x^{\Ttran}Bx)^{2}}{x^{\Ttran}BA^{-1}Bx}\Bigr)^{1/2}}{\norm{x}_{A}}\leq \sqrt{1-\kappa_{\nu}^{-1}},
		\end{aligned}
	\end{equation*}
	where we use \cref{defnu} in the last inequality.
\end{proof}

In order to further investigate the local property of $\cos\vartheta$, we give a decomposition.

\begin{theorem}
	\label{thmvartheta}
	Suppose that $\lambda_{1}\leq\rho< (\lambda_{1}+\lambda_{2})/2$, and let 
	\begin{equation}
		\label{defveps}
		\varepsilon=\frac{\lambda_{1}^{-1}-\rho^{-1}}{\lambda_{1}^{-1}-\lambda_{2}^{-1}}.
	\end{equation}
	The leading angle $\vartheta$ defined in \cref{defvartheta} satisfies
	\begin{equation}
		\label{estLA}
		\begin{aligned}
			\cos\vartheta &\leq \frac{\norm{B^{-1}Au_{1}-\nu_{1} u_{1}}_{A}}{\sqrt{\lambda_{1}}\nu_{\min}}+\Bigl(\kappa_{\nu}-\frac{2\nu_{1}}{\nu_{\min}}+\frac{\nu_{1}^{2}}{\nu_{\min}^{2}}\Bigr)^{1/2}\varepsilon^{1/2}\\ 
			&\leq \frac{\norm{B^{-1}Au_{1}-\nu_{1} u_{1}}_{A}}{\sqrt{\lambda_{1}}\nu_{\min}}+\sqrt{\kappa_{\nu}\varepsilon},
		\end{aligned}
	\end{equation}
	where $(\lambda_{1},u_{1})$ is the smallest eigenvalue of $(A,M)$ and its associated eigenvector with $\norm{u_{1}}_{M}=1$, $\nu_{\min}$ and $\nu_{\max}$ are the smallest and largest eigenvalue of $B^{-1}A$, respectively, and
	\begin{equation}
		\label{defnu1}
		\nu_{1} = \argmin_{\alpha\in\R}\norm{B^{-1}Au_{1}-\alpha u_{1}}_{A}=  \frac{u_{1}^{\Ttran}AB^{-1}Au_{1}}{u_{1}^{\Ttran}Au_{1}}.
	\end{equation}
\end{theorem}

\cref{thmvartheta} provides a clear illustration of the leading angle. For the first term in \cref{estLA}, note that $\nu_{1}$ depends only on the behavior of $B^{-1}$ in the one-dimensional space spanned by $Au_{1}$. Since $u_{1}$ is an eigenvector of $(A,M)$, if $B$ approximate $A$ well in $\spa{Au_{1}}$, then $\nu_{1}$ can be bounded independent of $\kappa_{\nu}$. The term $\sqrt{\kappa_{\nu}\varepsilon}$  reflects the global properties of the preconditioner $B$. Compared with \cref{varthetabdd}, the factor $\varepsilon$ mitigates the impact of the ratio $\kappa_{\nu}$ when $\rho$ is close to $\lambda_{1}$.

\begin{proof}[Proof of \cref{thmvartheta}]
	For any $M$-unit $x$ satisfying $\Rq(x)\leq \rho$, and $v^{\Ttran}Bx=0$, by the Cauchy--Schwarz inequality,
	\begin{equation*}
		\frac{\abs{x^{\Ttran}Av}}{\norm{x}_{A}\norm{v}_{A}} 
		= \min_{\alpha\in\R}\frac{\bigabs{x^{\Ttran}(A-\alpha B)v}}{\norm{x}_{A}\norm{v}_{A}}
		\leq \frac{\norm{B^{-1}Ax-\nu_{1}x}_{B}}{\norm{x}_{A}}\frac{\norm{v}_{B}}{\norm{v}_{A}}\leq  \frac{\norm{B^{-1}Ax-\nu_{1}x}_{B}}{\sqrt{\nu_{\min}}\norm{x}_{A}}.
	\end{equation*}
	Let $\alpha=x^{\Ttran}Mu_{1}$. Using the triangular inequality and $\abs{\alpha}\leq 1$, we have 
	\begin{equation*}
		\begin{aligned}
			\norm{B^{-1}Ax-\nu_{1}x}_{B} &= \norm{(B^{-1}A-\nu_{1}I)(x-\alpha u_{1})}_{B}+\abs{\alpha}\norm{(B^{-1}A-\nu_{1}I)u_{1}}_{B}\\ 
			&\leq \norm{(B^{-1}A-\nu_{1}I)(x-\alpha u_{1})}_{B}+\frac{\norm{(B^{-1}A-\nu_{1}I)u_{1}}_{A}}{\sqrt{\nu_{\min}}}.
		\end{aligned}
	\end{equation*}
	For the first term, by the variational property of the largest eigenvalue, 
	\begin{equation*}
		\begin{aligned}
			\frac{\norm{(B^{-1}A-\nu_{1}I)(x-\alpha u_{1})}_{B}^{2}}{\norm{x-\alpha u_{1}}_{A}^{2}} &= \frac{(x-\alpha u_{1})^{\Ttran}(AB^{-1}A-2\nu_{1}A+\nu_{1}^{2}B)(x-\alpha u_{1})}{(x-\alpha u_{1})^{\Ttran}A(x-\alpha u_{1})}\\ 
			&\leq \lambda_{\max}\bigl(A^{-1}(AB^{-1}A-2\nu_{1}A+\nu_{1}^{2}B)\bigr)\\ 
			&\leq \lambda_{\max}(B^{-1}A)-2\nu_{1}+\nu_{1}^{2}\lambda_{\max}(A^{-1}B)\\ 
			&\leq \nu_{\max}-2\nu_{1}+\frac{\nu_{1}^{2}}{\nu_{\min}},		
		\end{aligned}
	\end{equation*}
	where we use \cref{defnu} in the last inequality. Note that $\Rq(x)\geq \alpha^{2}\lambda_{1}+(1-\alpha^{2})\lambda_{2}$. We have a lower bound of $\alpha^{2}$ as 
	\begin{equation}
		\label{estxMu}
		\frac{\abs{x^{\Ttran}Mu_{1}}^{2}}{\norm{x}_{M}^{2}\norm{u_{1}}_{M}^{2}}=\alpha^{2}\geq \frac{\lambda_{2}-\Rq(x)}{\lambda_{2}-\lambda_{1}}.
	\end{equation} 
	Then using $\norm{x}_{A}^{2}=\Rq(x)\leq \rho$, we know
	\begin{equation*}
		\label{estLA2}
		\frac{\norm{x-\alpha u_{1}}_{A}^{2}}{\norm{x}_{A}^{2}} =1-\frac{\alpha^{2}\lambda_{1}}{\norm{x}_{A}^{2}} 
		\leq \frac{\lambda_{2}(\Rq(x)-\lambda_{1})}{\Rq(x)(\lambda_{2}-\lambda_{1})}
		\leq \frac{\lambda_{1}^{-1}-\rho^{-1}}{\lambda_{1}^{-1}-\lambda_{2}^{-1}}=\varepsilon.
	\end{equation*}
	The theorem is proved by combining all inequalities above.
\end{proof}
\begin{remark}
	In order to eliminate the matrix multiplication of $B$, which is hard to analyze for practical preconditioners, we scarify the order of $\kappa_{\nu}$ in the proof of \cref{thmvartheta}. 
\end{remark}

One motivation for developing the leading angle is to establish a lower bound for $\Rq(v)$, where $v$ is in the $B$-orthogonal complement.
The following theorem introduces a new parameter,  $\varrho$, which generalizes \cref{estRqMv} in the context of preconditioning, to address this issue.

\begin{theorem}
	\label{defvarrho}
	Using the notations in \cref{defvartheta}, let
	\begin{equation*}
		\varrho(B,\rho;\Rq) = \inf_{\Rq(x)\leq\rho}
		\inf_{v^{\Ttran}Bx=0} \Rq(v).
	\end{equation*}
	Denote $\varrho=\varrho(B,\rho;\Rq)$, then
	\begin{equation*}
		\frac{\varrho^{-1}-\lambda_{2}^{-1}}{\lambda_{1}^{-1}-\lambda_{2}^{-1}}\leq (\cos\vartheta+\sqrt{\varepsilon})^{2},
	\end{equation*}
	where $\varepsilon$ is defined in \cref{defveps}.
\end{theorem}
\begin{proof}
	Without loss of generality, assume $\norm{x}_{A}=\norm{v}_{A}=1$. Let $\{\widehat{u}_{i}\}$ be $A$-unit eigenvectors of $(A,M)$, \ie $\widehat{u}_{i}=u_{i}/\sqrt{\lambda_{i}}$. Suppose $x$ and $v$ have eigendecompositions
	\begin{equation*}
		x=\sum_{i=1}^{n}c_{i}\widehat{u}_{i}\quad \text{and}\quad v=\sum_{i=1}^{n}\widetilde{c}_{i}\widehat{u}_{i},\quad \text{where }\sum_{i=1}^{n}c_{i}^{2}=\sum_{i=1}^{n}\widetilde{c}_{i}^{2}=1.
	\end{equation*}
	According to \cref{defvartheta}, we have $\abs{\sum_{i=1}^{n}c_{i}\widetilde{c}_{i}} = \abs{x^{\Ttran}Av} \leq \cos\vartheta$.
	Let $\abs{c_{1}}=\cos\theta$ and $\abs{\widetilde{c}_{1}}=\cos\widetilde{\theta}$, where $\theta$ and $\widetilde{\theta}$ are both in $[0,\pi/2]$, we know that
	\begin{equation*}
		\cos\theta\cos\widetilde{\theta}-\cos\vartheta
		\leq \Bigabs{\sum_{i=2}^{n}c_{i}\widetilde{c}_{i}}
		\leq \Bigl(\sum_{i=2}^{n}c_{i}^{2}\Bigr)^{1/2}\Bigl(\sum_{i=2}^{n}\widetilde{c}_{i}^{2}\Bigr)^{1/2}
		=\sin\theta\sin\widetilde{\theta}.
	\end{equation*}
	Rearranging this inequality we have
	\begin{equation*}
		\cos\vartheta \geq \cos\theta\cos\widetilde{\theta}-\sin\theta\sin\widetilde{\theta}=\cos(\theta+\widetilde{\theta}),
	\end{equation*}
	which implies $\vartheta\leq \widetilde{\theta}+\theta$.
	Then  $\abs{\widetilde{c}_{1}}$ is bounded by
	\begin{equation*}
		\abs{\widetilde{c}_{1}} = \cos\widetilde{\theta} \leq  \cos( \vartheta-\theta)  \leq \abs{\cos\vartheta}+\abs{\sin\theta}= \cos\vartheta+\sqrt{1-c_{1}^{2}}.
	\end{equation*}
	According to the bound for $M$-angle between $u_{1}$ and $x$ in \cref{estxMu},
	\begin{equation*}
		c_{1}^{2} = (x^{\Ttran}A\widehat{u}_{1})^{2} =  \frac{(x^{\Ttran}A\widehat{u}_{1})^{2}}{\norm{x}_{A}^{2}\norm{\widehat{u}_{1}}_{A}^{2}}
		=\frac{(x^{\Ttran}Mu_{1})^{2}}{\norm{x}_{M}^{2}\norm{u_{1}}_{M}^{2}}\frac{\lambda_{1}}{\Rq(x)}
		\geq \frac{\lambda_{1}(\lambda_{2}-\rho)}{\rho(\lambda_{2}-\lambda_{1})}=1-\varepsilon.
	\end{equation*}
	Combining these two inequalities above, we know $\abs{\widetilde{c}_{1}}\leq \cos\vartheta+\sqrt{\varepsilon}$. Substituting 
	\begin{equation*}
		\frac{1}{\Rq(v)}
		=\sum_{i=1}^{n}\frac{\widetilde{c}_{i}^{2}}{\lambda_{i}}
		\leq \frac{\widetilde{c}_{1}^{2}}{\lambda_{1}}+\frac{1-\widetilde{c}_{1}^{2}}{\lambda_{2}}
        \leq \lambda_{2}^{-1}+(\lambda_{1}^{-1}-\lambda_{2}^{-1})(\cos\vartheta+\sqrt{\varepsilon})^{2}
	\end{equation*}
        as an upper bound of $\varrho^{-1}$ into $(\varrho^{-1}-\lambda_{2}^{-1})/(\lambda_{1}^{-1}-\lambda_{2}^{-1})$, we finish the proof.
\end{proof}

\begin{corollary}
	\label{asprho}
	The relationship $\varrho>\rho$ holds when
	\begin{equation*}
		\cos\vartheta < \frac{(\rho^{-1}-\lambda_{2}^{-1})^{1/2}-(\lambda_{1}^{-1}-\rho^{-1})^{1/2}}{(\lambda_{1}^{-1}-\lambda_{2}^{-1})^{1/2}}.
	\end{equation*}
\end{corollary}

The parameter $\varrho$ defined in \cref{defvarrho} provides a lower bound for $\Rq_{B}(v)$ in \cref{Hess}. Another term to address is $\norm{x}_{A_{B}}$, which is mathematically equivalent to $x^{\Ttran}Ax/x^{\Ttran}Bx$. We need to demonstrate that, in a neighborhood of $u_{1}$, the ratio $x^{\Ttran}Ax/x^{\Ttran}Bx$ is close to $u_{1}^{\Ttran}Au_{1}/u_{1}^{\Ttran}Bu_{1}$. This can be intuitively understood through continuity. The next theorem quantifies this relationship.

\begin{theorem}
	\label{estu1} 
	Suppose $\lambda_{1}\leq\rho< (\lambda_{1}+\lambda_{2})/2$, let
	\begin{equation*} 
		\varsigma_{\max}(B,\rho;\Rq) \defi \max_{\Rq(x)\leq \rho} \frac{x^{\Ttran}Ax}{x^{\Ttran}Bx},
		\quad\text{and}\quad
		\varsigma_{\min}(B,\rho;\Rq) \defi\min_{\Rq(x)\leq \rho} \dfrac{x^{\Ttran}Ax}{x^{\Ttran}Bx}.
	\end{equation*}
	Denote $\varsigma_{\max}= \varsigma_{\max}(B,\rho;\Rq)$ and $\varsigma_{\min}=\varsigma_{\min}(B,\rho;\Rq)$. Let
	\begin{equation*}
		\varepsilon_{*} = \Bigl(\frac{8\lambda_{2}}{\lambda_{2}-\lambda_{1}}\cos\vartheta + \sqrt{2\varepsilon}\Bigr)^{2}
		\quad\text{and}\quad
		\sigma = \frac{u_{1}^{\Ttran}Au_{1}}{u_{1}^{\Ttran}Bu_{1}}, 
	\end{equation*}
	where $\vartheta=\vartheta(B,\rho;\Rq)$ and $\varepsilon$ are defined in \cref{defvartheta,thmvartheta} respectively.
	Assume that $\varepsilon_{*}\leq 1/2$, then 
	\begin{equation}
		\frac{\sigma(1-2\sqrt{\varepsilon_{*}}\cos\vartheta)}{1+(\sigma\nu_{\min}^{-1}-1)\varepsilon_{*}}\leq \varsigma_{\min}\leq \varsigma_{\max}\leq \frac{\sigma(1+2\sqrt{\varepsilon_{*}}\cos\vartheta)}{1-(1-\sigma\nu_{\max}^{-1})\varepsilon_{*}}.
	\end{equation}
\end{theorem}
\begin{proof}
	See \cref{appestu1}.
\end{proof}

\begin{remark}
	Suppose $\cos\vartheta=\order(\sqrt{\varepsilon})$, then
	\begin{equation*}
		\varsigma_{\max}/\varsigma_{\min} \leq 1+\order(\sigma\nu_{\min}^{-1}\varepsilon).	
	\end{equation*}
	Compared with the direct bound $\varsigma_{\max}/\varsigma_{\min}\leq \kappa_{\nu}$,
	our new bound is significantly better.
\end{remark}

At the end of this part, we show the invariance of $\vartheta$, $\varrho$, $\varsigma_{\min}$ and $\varsigma_{\max}$ under similarity transformation.
\begin{proposition}
	Recall the $\Rq_{B}(x)$ defined in \cref{preconOpt} as 
	\begin{equation*}
		\Rq_{B}(x) = \frac{x^{\Ttran}A_{B}x}{x^{\Ttran}M_{B}x}, \quad\text{where}\quad A_{B}=B^{-1/2}AB^{-1/2}
        \quad\text{and}\quad 
        M_{B}=B^{-1/2}MB^{-1/2}.
	\end{equation*}
	It holds that 
	\begin{equation*}
		\begin{aligned}
			\vartheta(B,\rho;\Rq) &= \vartheta(I,\rho;\Rq_{B}),\quad
		\varrho(B,\rho;\Rq) = \varrho(I,\rho;\Rq_{B}),\\ 
		\varsigma_{\min}(B,\rho;\Rq) &= \varsigma_{\min}(I,\rho;\Rq_{B}),\quad
		\varsigma_{\max}(B,\rho;\Rq) = \varsigma_{\max}(I,\rho;\Rq_{B}).
		\end{aligned}
	\end{equation*}
\end{proposition}
\begin{proof}
	We just show $\vartheta(B,\rho;\Rq) = \vartheta(I,\rho;\Rq_{B})$ since the proof are essentially same. By definition, 
	\begin{equation*}
		\begin{aligned}
			\vartheta(B,\rho;\Rq) &= \inf_{\Rq(x)\leq\rho} \inf_{v^{\Ttran}Bx=0} 
				\arccos\Bigl(\frac{\abs{v^{\Ttran}Ax}}{\norm{v}_{A}\norm{x}_{A}}\Bigr)\\ 
			&=\inf_{\Rq_{B}(B^{1/2}x)\leq\rho} \inf_{(B^{1/2}v)^{\Ttran}(B^{1/2}x)=0} 
			\arccos\Bigl(\frac{\abs{(B^{1/2}v)^{\Ttran}A_{B}(B^{1/2}x)}}{\norm{B^{1/2}v}_{A_{B}}\norm{B^{1/2}x}_{A_{B}}}\Bigr)\\ 
			&=\vartheta(I,\rho;\Rq_{B}).
		\end{aligned}
	\end{equation*}
\end{proof}

\subsection{Analysis on $B$-sphere: with preconditioning}

With the new instruments developed in the last subsection, we can now analyze the geodesic convexity and Lipschitz smoothness of $\Rq_{B}$ on $\mathcal{X}_{B}$.

\begin{theorem}
	\label{lemcvxfunP}
	Suppose $\lambda_{1}\leq\rho<(\lambda_{1}+\lambda_{2})/2$, let $\Rq_{B}(x)$ be the Rayleigh quotient defined in \cref{preconOpt}, denote
	\begin{equation*}
		\begin{aligned}
			\vartheta &= \vartheta(I,\rho;\Rq_{B}),\quad
		\varrho = \varrho(I,\rho;\Rq_{B}),\\ 
		\varsigma_{\min} &= \varsigma_{\min}(I,\rho;\Rq_{B}),\quad
		\varsigma_{\max} = \varsigma_{\max}(I,\rho;\Rq_{B}).
		\end{aligned}
	\end{equation*}
	Under the condition $\varrho>\rho$, for any $x\in\mathcal{X}_{B}$, the following inequality holds:
	\begin{equation}
		\label{defmuLP}
		\mu_{B}\norm{v}^{2} \leq
		\dual{v,\nabla^{2}\Rq_{B}(x)[v]}\leq
		L_{B}\norm{v}^{2},
	\end{equation}
	where
	\begin{equation}
		\label{defLmu}
		\begin{aligned}
			\mu_{B}         & = \frac{2\nu_{\min}\lambda_{1}}{\varsigma_{\max}}\Bigl(1-\frac{\rho}{\varrho}\Bigr)-C_{\mathcal{X}},\quad
			L_{B} = \frac{2\nu_{\max}\rho}{\varsigma_{\min}}\Bigl(1-\frac{\lambda_{1}}{\lambda_{n}}\Bigr)+C_{\mathcal{X}},                  \\
			C_{\mathcal{X}} & = \frac{8\nu_{\max}\rho}{\varsigma_{\min}}
			\biggl(
			\frac{\rho-\lambda_{1}}{\lambda_{1}}
			+\Bigl(\frac{\rho-\lambda_{1}}{\lambda_{1}}\Bigr)^{1/2}\cos\vartheta
			\biggr),\quad
			\sigma = \frac{u_{1}^{\Ttran}Au_{1}}{u_{1}^{\Ttran}Bu_{1}},
		\end{aligned}
	\end{equation}
	and $\nu_{\min}$ and $\nu_{\max}$ are the smallest and largest eigenvalues of $A_{B}$, respectively.
\end{theorem}
\begin{proof}
	Without loss of generality, assume $\norm{v}=1$, then\footnote{Details are provided in \cref{appHess} in \cref{appRO}.}
	\begin{equation*}
		\dual{v,\nabla^{2}\Rq_{B}(x)[v]} = J_{1}-J_{2},
	\end{equation*}
	where, denoting $r = A_{B}x-\Rq_{B}(x)M_{B}x$, 
	\begin{equation*}
		J_{1}  = \frac{2\norm{v}_{A_{B}}^{2}\Rq_{B}(x)}{\norm{x}_{A_{B}}^{2}}\Bigl(1-\frac{\Rq_{B}(x)}{\Rq_{B}(v)}\Bigr)
        \quad\text{and}\quad 
		J_{2}  = \frac{8}{\norm{x}_{M_{B}}^{2}}\cdot v^{\Ttran}r\cdot \frac{v^{\Ttran}M_{B}x}{\norm{x}_{M_{B}}^{2}}.
	\end{equation*}

	Under the condition $\varrho>\rho$, the first term $J_{1}$ can be bounded with
	\begin{equation}
		\label{estI1}
		\frac{2\nu_{\min}\lambda_{1}}{\varsigma_{\max}}\Bigl(1-\frac{\rho}{\varrho}\Bigr)
		\leq J_{1} \leq
		\frac{2\nu_{\max}\rho}{\varsigma_{\min}}\Bigl(1-\frac{\lambda_{1}}{\lambda_{n}}\Bigr),
	\end{equation}
	by $\lambda_{1}\leq \Rq_{B}(x)\leq \rho$, $\varrho\leq \Rq_{B}(v)\leq \lambda_{n}$, $\nu_{\min}\leq \norm{v}_{A_{B}}^{2}\leq \nu_{\max}$ and the estimation for $\norm{x}_{A_{B}}^{2}$ in \cref{estu1}.

	For the second term $J_{2}$, by the Cauchy--Schwarz inequality, we know 
	\begin{equation*}
		\abs{v^{\Ttran}r}\leq \norm{v}_{A_{B}}\norm{r}_{A_{B}^{-1}}.
	\end{equation*}
	According to the definition of $\vartheta$, we obtain
	\begin{equation*}
		\begin{aligned}
			\frac{\abs{v^{\Ttran}M_{B}x}}{\norm{x}_{M_{B}}^{2}} & =\frac{\abs{v^{\Ttran}\Rq_{B}(x)M_{B}x}}{\norm{x}_{A_{B}}^{2}}
			=\frac{\abs{v^{\Ttran}(r-A_{B}x)}}{\norm{x}_{A_{B}}^{2}}
			\leq \frac{\abs{v^{\Ttran}r}+\abs{v^{\Ttran}A_{B}x}}{\norm{x}_{A_{B}}^{2}}\\
			& \leq \frac{\norm{v}_{A_{B}}\norm{r}_{A_{B}^{-1}}+\norm{v}_{A_{B}}\norm{x}_{A_{B}}\cos\vartheta}{\norm{x}_{A_{B}}^{2}}.
		\end{aligned}
	\end{equation*} 
	Combining two inequalities above, we have
	\begin{equation}
 \label{estJ20}
         \begin{aligned}
             \abs{J_{2}}& \leq \frac{8}{\norm{x}_{M_{B}}^{2}}\cdot \norm{v}_{A_{B}}\norm{r}_{A_{B}^{-1}}\cdot \frac{\norm{v}_{A_{B}}\norm{r}_{A_{B}}^{-1}+\norm{v}_{A_{B}}\norm{x}_{A_{B}}\cos\vartheta}{\norm{x}_{A_{B}}^{2}} \\ 
             &\leq \frac{8\nu_{\max}}{\varsigma_{\min}}
        			\biggl(\frac{\norm{r}_{A_{B}^{-1}}^{2}}{\norm{x}_{M_{B}}^{2}}+\frac{\norm{r}_{A_{B}^{-1}}\norm{x}_{A_{B}}\cos\vartheta}{\norm{x}_{M_{B}}^{2}}
        			\biggr),    
         \end{aligned}    
	\end{equation}
 where we use $\norm{v}_{A_{B}}^{2}\leq \nu_{\max}$ and $\norm{x}_{A_{B}}^{2}\geq \varsigma_{\min}$ in the last inequality. 
	 Note that
	\begin{equation*}
		\begin{aligned}
			\frac{x^{\Ttran}M_{B}A_{B}^{-1}M_{B}x}{x^{\Ttran}M_{B}x} 
			=\frac{(M_{B}x)^{\Ttran}A_{B}^{-1}(M_{B}x)}{(M_{B}x)^{\Ttran}M_{B}^{-1}(M_{B}x)}\leq \frac{1}{\lambda_{1}}.
		\end{aligned}
	\end{equation*}
	We can provide a bound for the residual $r$ as 
	\begin{equation}
 \label{estr}
		\frac{\norm{r}_{A_{B}^{-1}}^{2}}{\norm{x}_{M_{B}}^{2}} = \Rq_{B}(x)\Bigl(\Rq_{B}(x)\frac{x^{\Ttran}M_{B}A_{B}^{-1}M_{B}x}{x^{\Ttran}M_{B}x}-1\Bigr)
		\leq \frac{\rho(\rho-\lambda_{1})}{\lambda_{1}},
	\end{equation}
	where $\Rq_{B}(x)\leq \rho$ is used.
	Combining \cref{estJ20,estr}, we have
	\begin{equation}
		\label{estI2}
		\abs{J_{2}}\leq \frac{8\nu_{\max}\rho}{\varsigma_{\min}}
		\biggl(
		\frac{\rho-\lambda_{1}}{\lambda_{1}}
		+\Bigl(\frac{\rho-\lambda_{1}}{\lambda_{1}}\Bigr)^{1/2}\cos\vartheta
		\biggr).
	\end{equation}
	Then the theorem is proved by \cref{estI1,estI2}.
\end{proof}

\begin{corollary}
	\label{rmkPEVP}
	When $\varepsilon=(\lambda_{1}^{-1}-\rho^{-1})/(\lambda_{1}^{-1}-\lambda_{2}^{-1})$ is sufficiently small, the condition $\varrho>\rho$ holds due to \cref{asprho}. Moreover, suppose $\cos\vartheta$ is also small, which guarantees the condition $\varepsilon_{*}\leq 1/2$ in \cref{estu1}, we have
	\begin{equation*}
		\begin{aligned}
			\mu_{B} &\geq  \frac{2\nu_{\min}\lambda_{1}}{\sigma}\Bigl(1-\frac{\lambda_{1}}{\lambda_{2}}\Bigr)\Bigl(1-\order\bigl(\kappa_{\nu}(\gap^{-2}\cos^{2}\vartheta+\varepsilon)\bigr)\Bigr),\\ 
			L_{B} &\leq  \frac{2\nu_{\max}\lambda_{1}}{\sigma}\Bigl(1-\frac{\lambda_{1}}{\lambda_{n}}\Bigr)\Bigl(1+\order\bigl(\sigma\nu_{\min}^{-1}(\gap^{-2}\cos^{2}\vartheta+\varepsilon)\bigr)\Bigr),
		\end{aligned}
	\end{equation*}
	where $\gap=(\lambda_{2}-\lambda_{1})/\lambda_{2}$. This means $\Rq_{B}$ is $\mu_{B}$-geodesically convex and $L_{B}$-Lipschitz smooth. Furthermore, the condition number of the Riemannian optimization problem after preconditioning \cref{preconOpt} is
	\begin{equation*}
		\kappa_{B} = \frac{L_{B}}{\mu_{B}} \leq \kappa_{\nu}\frac{1-\lambda_{1}/\lambda_{n}}{1-\lambda_{1}/\lambda_{2}}\Bigl(1+\order\bigl(\kappa_{\nu}(\gap^{-2}\cos^{2}\vartheta+\varepsilon)\bigr)\Bigr).
	\end{equation*}
	Compared with the condition number of \cref{stdopt} in \cref{rmkEVP}, \ie
	\begin{equation*}
		\kappa = \frac{\lambda_{n}}{\lambda_{2}}\frac{1-\lambda_{1}/\lambda_{n}}{1-\lambda_{1}/\lambda_{2}}\bigl(1+\order(\varepsilon)\bigr),
	\end{equation*}
	when $B$ is a good preconditioner for $A$, \ie $\kappa_{\nu}=\nu_{\max}/\nu_{\min}$ and $\cos\vartheta$ are small, the condition number $\kappa_{B}$ is significantly improved.
\end{corollary}
\begin{proof}
	See \cref{apprmkPEVP}.
\end{proof}
\begin{remark}
	The condition number $\kappa_{B}$ plays a crucial role in the convergence of gradient type methods. For gradient methods and accelerated gradient methods, the rate of convergence are $1-c\kappa_{B}^{-1}$ and $1-c\kappa_{B}^{-1/2}$, respectively. 
\end{remark}

\section{Convex Riemannian optimization}
In this section, we consider a general convex Riemannian optimization problem
\begin{equation}
	\label{optmanifolds}
	x_{*}=\argmin_{x\subset\mathcal{M}} f(x),
\end{equation}
where $\mathcal{M}$ is a Riemannian manifold and $f\colon \mathcal{M}\mapsto\R$ is a smooth function. In order to save notations, we reuse the notations in the previous section. Here we list our basic assumption about \cref{optmanifolds}.
\begin{Assumption}
	\label{aspX}
	Suppose $\mathcal{X}\subset\mathcal{M}$ is a geodesically convex region, where the sectional curvature of $\mathcal{M}$ is bounded in $[-K,K]$. Assume $x_{*}\in\mathcal{X}$ and $f$ is $\mu$-geodesically convex and $L$-Lipschitz smooth on $\mathcal{X}$. Moreover, let $\mathcal{B}_{R}$ be the closed ball on $\mathcal{M}$ defined as
	\begin{equation*}
		\mathcal{B}_{R}\defi \{x\in\mathcal{M}\mid \norm{\log_{x_{*}}(x)}\leq  R\}.
	\end{equation*}
	Assume that for any $R>0$, if $\mathcal{B}_{R}\subset \mathcal{X}$, then $\mathcal{B}_{R}$ is geodesically convex.
\end{Assumption}
\begin{remark}
	Unlike most literature on Riemannian acceleration, only local geodesic convexity on $\mathcal{X}$ instead of $\mathcal{M}$ is assumed in \cref{aspX}, which aligns with our previous results on the preconditioned eigensystem in \cref{lemcvxfunP}. 
\end{remark}
\subsection{Riemannian accelerated gradient descent method}
The estimate sequence, which was first proposed by Nesterov in \cite{Nesterov1983}, is a traditional and popular technique for the analysis of momentum terms. Zhang and Sra generalized it into Riemannian manifolds, called weak estimate sequence, as follows.

\begin{definition}[{Weak estimate sequence, \cite[Definition~1]{Zhang2018}}]
	Let $f\colon \mathcal{M}\mapsto \R$, if there exists $0<\alpha<1$ and a sequence $\phi_{m}\colon \mathcal{M}\mapsto \R$ such that
	\begin{equation*}
		\phi_{m}(x_{*})-f(x_{*})\leq (1-\alpha)^{m} \bigl(\phi_{0}(x_{*})-f(x_{*})\bigr),
	\end{equation*}
	where $x_{*}$ is the solution of \cref{optmanifolds}, then $(\alpha,\{\phi_{m}\})$ is called a weak estimate sequence\footnote{Here we slightly simplify the definition in \cite{Zhang2018}. In their paper, $\alpha$ is allowed to depend on $m$.}.
\end{definition}

With the weak estimate sequence, Zhang and Sra proposed a Riemannian Accelerated Gradient Descent (RAGD) method, \ie \cite[Algorithm~2]{Zhang2018}, and analyzed its convergence. Here we slightly generalize it as follows. For initial point $v_{0}=x_{0}\in\mathcal{X}$, step size $h\leq 1/L$ and shrinkage parameter $\beta>0$, let
\begin{equation*}
	\alpha=\frac{\sqrt{\beta^{2}+4(1+\beta)\mu h}-\beta}{2},\quad \gamma=\frac{\alpha\mu}{\alpha+\beta}\quad \text{and}\quad \overline{\gamma}=(1+\beta)\gamma.
\end{equation*}
For $m=0,1,\dotsc$, consider the following recurrence:
\begin{equation}
	\label{GRAGD}
	\left\{
	\begin{aligned}
		 & y_{m}=\exp_{x_{m}}\Bigl(\dfrac{\alpha}{1+\alpha+\beta}\log_{x_{m}}(v_{m})\Bigr),                                            \\
		 & v_{m+1}=\exp_{y_{m}}\Bigl(\frac{1-\alpha}{1+\beta}\log_{y_{m}}(v_{m})-\frac{\alpha}{\overline{\gamma}}\grad f(y_{m})\Bigr), \\
		 & \text{update $x_{m+1}$ satisfying $f(x_{m+1})\leq f(y_{m})-\frac{h}{2}\norm{\grad f(y_{m})}^{2}$.}
	\end{aligned}
	\right.
\end{equation}

\begin{remark}
	In RAGD of \cite{Zhang2018}, the iterates $x_{m+1}$ are constructed as
\begin{equation*}
	x_{m+1}=\exp_{y_{m}}(-h\grad f(y_{m})).
\end{equation*}
Note that in their estimate sequence approach, the only requirement for $x_{m+1}$, in the proof of \cite[Lemma~6]{Zhang2018}, is that: 
\begin{equation*}
	f(x_{m+1})\leq f(y_{m})-\frac{h}{2}\norm{\grad f(y_{m})}^{2}\leq \phi_{m+1}^{*},
\end{equation*}
where $\phi_{m+1}^{*}$ is a given scalar independent of the choice of $x_{m}$. Thus theoretical convergence results in \cite{Zhang2018} can be generalized into \cref{GRAGD}, which is given in \cref{propestSeq}. For self-containess, we also provide a more detailed argument for its proof in \cref{apppropestSeq}.
\end{remark}

\begin{proposition}
	\label{propestSeq}
	For the iterates from \cref{GRAGD} with any shrinkage parameter $\beta>0$ and initial value $\phi_{0}^{*} = f(x_{0})$ and $\phi_{0}(x)=\phi_{0}^{*}+\frac{\gamma}{2}\norm{\log_{x_{0}}(x)}^{2}$, let
	\begin{equation*}
		\phi_{i+1}(x)=  \phi_{i+1}^{*}+\frac{\overline{\gamma}}{2}\norm{\log_{y_{i}}(x)-\log_{y_{i}}(v_{i+1})}^{2},
	\end{equation*}
	for $i\geq 0$, where
	\begin{equation*}
		\begin{aligned}
			\phi_{i+1}^{*} = & (1-\alpha)\phi_{i}^{*}+\alpha f(y_{i})-\frac{\alpha^{2}}{2\overline{\gamma}}\norm{\grad f(y_{i})}^{2}                               \\
			                 & +\frac{\alpha(1-\alpha)}{1+\beta}\Bigl(\frac{\mu}{2}\norm{\log_{y_{i}}(v_{i})}^{2}+\dual{\grad f(y_{i}),\log_{y_{i}}(v_{i})}\Bigr).
		\end{aligned}
	\end{equation*}
	Suppose that $\{x_{i}\}_{i=0}^{m+1}$, $\{y_{i}\}_{i=0}^{m}$ and $\{v_{i}\}_{i=0}^{m+1}$ are all in $\mathcal{X}$ and
	\begin{equation}
		\label{distortion}
		\norm{\log_{y_{i+1}}(x_{*})-\log_{y_{i+1}}(v_{i+1})}^{2}
		\leq (1+\beta) \norm{\log_{y_{i}}(x_{*})-\log_{y_{i}}(v_{i+1})}^{2}
	\end{equation}
	holds for $0\leq i\leq m-1$,
	then the following relationships hold:
	\begin{enumerate}
		\item $(\alpha,\{\phi_{i}\}_{i=0}^{m+1})$ is a weak estimate sequence.
		\item $f(x_{i})\leq \phi_{i}^{*}$ holds for all $0\leq i\leq m+1$.
		\item $\phi_{i+1}(x_{*})-f(x_{*}) \leq (1-\alpha) \bigl(\phi_{i}(x_{*})-f(x_{*})\bigr)$ holds for all $0\leq i\leq m$.
		\item The iterates $\{x_{i}\}_{i=0}^{m+1}$ from \cref{GRAGD} satisfy
		\begin{equation}
			\label{mainprop}
			\begin{aligned}
				f(x_{i})-f(x_{*}) & \leq \phi_{i}(x_{*})-f(x_{*})\leq (1-\alpha)^{i+1} \bigl(\phi_{0}(x_{*})-f(x_{*})\bigr)       \\
									& =(1-\alpha)^{i} \Bigl(f(x_{0})-f(x_{*})+\frac{\gamma}{2}\norm{\log_{x_{0}}(x_{*})}^{2}\Bigr).
			\end{aligned}
		\end{equation}
	\end{enumerate}
\end{proposition}

\subsection{Locally optimal Riemannian accelerated gradient method}
In this part, we propose a new Riemannian acceleration scheme, called Locally Optimal Riemannian Accelerated Gradient (LORAG) as \cref{algo}, by combining the RAGD method with the locally optimal technique from LOBPCG \cite{Knyazev2001}.
\begin{algorithm2e}[H]
	\caption{Locally Optimal Riemannian Accelerated Gradient method}
	\label{algo}
	\KwIn{Initial point $x_{0}$, parameters $\beta$, $\mu$ and $L$.}
	Set $0<\alpha\leq \frac{\sqrt{\beta^{2}+4(1+\beta)\mu/L }-\beta}{2}$, $\gamma=\frac{\alpha\mu}{\alpha+\beta}$, $\overline{\gamma}=(1+\beta)\gamma$, and $v_{0}=x_{0}$\;
	\For{$m=0,1,2,\dotsc,$}{
		Update $y_{m}$ and $v_{m+1}$ by \cref{GRAGD}\;
		Find $x_{m+1}$ such that $x_{m+1}=\argmin\limits_{g\in\spa{\grad f(y_{m}),\log_{y_{m}}(x_{m})}}f\bigl(\exp_{y_{m}}(g)\bigr)$\;
	}
\end{algorithm2e}

The following lemma shows an advantage of the LORAG method over RAGD methods, specifically in terms of the monotonicity of the function values $f(x_{m})$.

\begin{lemma}
	\label{lemxk}
	For any $\{x_{m}\}$ from \cref{algo}, $f(x_{m+1})\leq f(x_{m})\leq \dotsb\leq f(x_{0})$.
\end{lemma}
\begin{proof}
	By the locally optimal step for $x_{m+1}$ in \cref{algo}, we have
	\begin{equation*}
		f(x_{m+1})\leq f\Bigl(\exp_{y_{m}}\bigl(\log_{y_{m}}(x_{m})\bigr)\Bigr)=f(x_{m})\leq \dotsb \leq f(x_{0}).
	\end{equation*}
	which finishes the proof.
\end{proof}

Besides the monotonicity, we are also interested in the rate of convergence for the LORAG method.
In contrast to the majority of existing Riemannian Nesterov methods, such as that described in \cite[Assumption~4]{Zhang2018}, which rely on the assumption that all iterates lie within the region $\mathcal{X}$ where the function $f$ is both geodesically convex and Lipschitz smooth, we explore methods to achieve convergence without the prior assumption that iterates are confined to this region.
As demonstrated in \cref{thmLORAG}, the sole assumption required is that the initial point, $x_{0}$, is sufficiently close to the solution, $x_{*}$. This assumption is particularly user-friendly for a diverse range of real-world problems, including eigenvalue computations as discussed in last section. This is because it is more straightforward to verify the condition of the initial point than that of all iterates, which are typically unavailable for verification prior to their computation.

The basic idea is to demonstrate that when the initial point $x_{0}$ is sufficiently close to $x_{*}$, the minimum in $\mathcal{M}$, all iterates $(y_{m}, x_{m+1}, v_{m+1})$ will remain within $\mathcal{X}$. A similar approach for Euclidean space has been discussed in \cite{Shao2023a}. From the perspective of Riemannian manifolds, the framework is comparable, although the technical details are more challenging.

\begin{theorem}
	\label{thmLORAG}
	Under \cref{aspX}, let 
	\begin{equation*}
		\begin{aligned}
			E_{0} & = \phi_{0}(x_{*})-f(x_{*}) = f(x_{0})-f(x_{*})+\frac{\gamma}{2}\norm{\log_{x_{0}}(x_{*})}^{2},\\ 
			R_{1} & = \Bigl(\frac{2E_{0}}{\gamma}\Bigr)^{1/2}, \quad \text{and}\quad R_{2}=\max \Bigl\{2R_{1}, \Bigl(1+\frac{\alpha L}{(1+\beta)\gamma}\Bigr)R_{1}\Bigr\}.
		\end{aligned}
	\end{equation*}
	Suppose that $(1+2KR_{1}^{2})^{2}\leq 1+\beta$ and 
	\begin{equation}
		\label{invarset}
		\{x\in\mathcal{M}\mid f(x)\leq f(x_{0}) \}\subset \mathcal{B}_{R_{1}}
		\subset \mathcal{B}_{R_{2}} \subset \mathcal{X},
	\end{equation}
	then the iterates $\{x_{m}\}$ from \cref{algo} lie in $\mathcal{B}_{R_{1}}\subset \mathcal{X}$.
\end{theorem}
\begin{proof}
	See \cref{appthmLORAG}
\end{proof}

Combining \cref{propestSeq,thmLORAG}, we establish the convergence of \cref{algo} as follows.
\begin{corollary}
	\label{corLORAG}
	With notations and conditions \cref{invarset} in \cref{thmLORAG}, the iterates from \cref{algo} satisfy
	\begin{align}
		\label{corLORAG1} & f(x_{m})\leq f(x_{m-1})\leq \dotsb \leq f(x_{0}), \\
		\label{corLORAG2} & f(x_{m})-f(x_{*})\leq (1-\alpha)^{m}E_{0}.
	\end{align}
\end{corollary}
\begin{proof}
	The monotonicity comes from \cref{lemxk}. For the rate of convergence, from \cref{thmLORAG} we know all iterates $(y_{m},x_{m+1},v_{m+1})$ are all in $\mathcal{B}_{R_{1}}\subset \mathcal{X}$, then the convergence is obtained from \cref{propestSeq}.
\end{proof}

At the end of this section, we provide some choices of hyper-parameters.

\begin{proposition}
	\label{rmkcoef}
	For function $f$ satisfying \cref{aspX}, assume $\kappa=L/\mu\geq 9$, let
	\begin{equation*}
		\begin{aligned}
			 & \beta = \frac{3}{2\sqrt{\kappa}-4},\quad
			\alpha = \frac{1}{2\sqrt{\kappa}},\quad
			\gamma = \frac{(\sqrt{\kappa}-2)\mu}{2(2\sqrt{\kappa}-1)},                                                \\
			 & R_{1} = \Bigl(\frac{2 E_{0}}{\gamma}\Bigr)^{1/2}\quad\text{and}\quad   R_{2}=(1+2\sqrt{\kappa})R_{1},
		\end{aligned}
	\end{equation*}
	where $E_{0}=f(x_{0})-f(x_{*})+\frac{\gamma}{2}\norm{\log_{x_{*}}(x_{0})}^{2}$. When $x_{0}$ is sufficiently close to $x_{*}$, \ie
	\begin{equation*}
		(1+KR_{1}^{2})^{2}\leq 1+\beta\quad\text{and}\quad \mathcal{B}(R_{2})\subset \mathcal{X},
	\end{equation*}
	then all conditions in \cref{thmLORAG} hold, and the rate of convergence is
	\begin{equation*}
		f(x_{m})-f(x_{*}) \leq \Bigl(1-\frac{1}{2\sqrt{\kappa}}\Bigr)^{m}E_{0}.
	\end{equation*}
	Moreover, in this situation, it holds that 
	\begin{equation*}
		R_{1}^{2}\leq \frac{22}{\mu}\bigl(f(x_{0})-f(x_{*})\bigr)
		\quad \text{and}\quad
		R_{2}^{2}\leq \frac{22}{\mu}(1+2\sqrt{\kappa})^{2}\bigl(f(x_{0})-f(x_{*})\bigr).
	\end{equation*}
\end{proposition}
\begin{proof}
	See \cref{apprmkcoef}.
\end{proof}

\begin{remark}
	The coefficients in \cref{rmkcoef} are not optimal. Actually, when $\beta$ is sufficiently small, $\alpha$ can be any positive scalar smaller than $\kappa^{-1/2}$.
\end{remark}
\section{Riemannian acceleration with preconditioning}

Now let us turn back to symmetric eigenvalue problems. As discussed in \cref{lemcvxfunP}, under some mild conditions, the symmetric eigenvalue problem with preconditioning is geodesically convex and Lipschitz smooth on manifolds, which means we can solve it by the LORAG method, \ie \cref{algo}.
One practical difficulty is the multiplication of the preconditioner $B$. For example, when computing $y_{m}$, we need to compute
\begin{equation*}
	\begin{aligned}
		\theta_{m} & = \dfrac{\alpha}{\alpha+\beta+1}\arccos(x_{m}^{\Ttran}Bv_{m}),           \\
		w_{m}      & = v_{m}-(x_{m}^{\Ttran}Bv_{m})x_{m},                                     \\
		y_{m}      & = \cos(\theta_{m})x_{m}+\sin(\theta_{m})\dfrac{w_{m}}{\norm{w_{m}}_{B}}.
	\end{aligned}
\end{equation*}
In most situation, the preconditioner $B$ is only available for solving linear system rather than multiplication. In order to overcome it, we use some ``co-iterates'':
\begin{equation}
	\label{coiterate}
	\widetilde{x}_{m}=Bx_{m},\quad \widetilde{v}_{m}=Bv_{m},\quad  \widetilde{y}_{m}=By_{m}.
\end{equation}
With \cref{coiterate}, we can compute the $B$ inner-product and norm by a standard inner-product of one iterate and one co-iterate. 
Specifically, the computation of $y_{m}$ and $\widetilde{y}_{m}$ become 
\begin{equation}
	\label{RAPy}
	\begin{aligned}
		\theta_{m} & = \dfrac{\alpha}{\alpha+\beta+1}\arccos(x_{m}^{\Ttran}\widetilde{v}_{m}),           \\
		w_{m}      & = v_{m}-(x_{m}^{\Ttran}\widetilde{v}_{m})x_{m},   \quad   \widetilde{w}_{m}  = \widetilde{v}_{m}-(x_{m}^{\Ttran}\widetilde{v}_{m})\widetilde{x}_{m},                                \\
		y_{m}      & = \cos(\theta_{m})x_{m}+\sin(\theta_{m})\dfrac{w_{m}}{\sqrt{w_{m}^{\Ttran}\widetilde{w}_{m}}},\\
		\widetilde{y}_{m}   &    = \cos(\theta_{m})\widetilde{x}_{m}+\sin(\theta_{m})\dfrac{\widetilde{w}_{m}}{\sqrt{w_{m}^{\Ttran}\widetilde{w}_{m}}}.
	\end{aligned}
\end{equation}
Similarly, we can compute $v_{m}$ and $\widetilde{v}_{m}$ by 
\begin{equation}
	\label{RAPv}
	\begin{aligned}
		\widetilde{g}_{m}   & = \frac{2}{y_{m}^{\Ttran}My_{m}}\Bigl(Ay_{m}-\frac{y_{m}^{\Ttran}Ay_{m}}{y_{m}^{\Ttran}My_{m}}My_{m}\Bigr),
		\quad g_{m}=B^{-1}\widetilde{g}_{m},\\
		p_{m}             & = v_{m}-(y_{m}^{\Ttran}\widetilde{v}_{m})y_{m},
		\quad \widetilde{p}_{m}=\widetilde{v}_{m}-(y_{m}^{\Ttran}\widetilde{v}_{m})\widetilde{y}_{m},         \\
		q_{m}             & = \frac{(1-\alpha)\theta_{m}}{\alpha\sqrt{p_{m}^{\Ttran}\widetilde{p}_{m}}}p_{m}-\frac{\alpha}{(1+\beta)\gamma}g_{m},             \quad 
		\widetilde{q}_{m}    = \frac{(1-\alpha)\theta_{m}}{\alpha\sqrt{p_{m}^{\Ttran}\widetilde{p}_{m}}}\widetilde{p}_{m}-\frac{\alpha}{(1+\beta)\gamma}\widetilde{g}_{m},                            \\
		v_{m+1}           & = \cos\sqrt{q_{m}^{\Ttran}\widetilde{q}_{m}}y_{m}+\frac{\sin \sqrt{q_{m}^{\Ttran}\widetilde{q}_{m}}}{\sqrt{q_{m}^{\Ttran}\widetilde{q}_{m}}}q_{m},   \\
		\widetilde{v}_{m+1}           & = \cos\sqrt{q_{m}^{\Ttran}\widetilde{q}_{m}}\widetilde{y}_{m}+\frac{\sin \sqrt{q_{m}^{\Ttran}\widetilde{q}_{m}}}{\sqrt{q_{m}^{\Ttran}\widetilde{q}_{m}}}\widetilde{q}_{m}. 
	\end{aligned}
\end{equation}
Another practical difficulty is to solve the locally optimal step as 
\begin{equation*}
	x_{m+1}=\argmin\limits_{g\in \mathcal{V}_{m}}\Rq\bigl(\exp_{y_{m}}(g)\bigr),\quad\text{where}\quad  \mathcal{V}_{m}=\spa{g_{m},\log_{y_{m}}(x_{m})}.
\end{equation*}
Note that for any $g\in\mathcal{V}_{m}$, by the homogeneity of Rayleigh quotient, \ie $\Rq(\alpha x)=\Rq(x)$ for any $\alpha\neq 0$, we have 
\begin{equation*}
	\Rq\bigl(\exp_{y_{m}}(g)\bigr) = \Rq\bigl((y_{m}+g\tan\norm{g}_{B})\cos\norm{g}_{B}\bigr) = \Rq(y_{m}+g\tan\norm{g}_{B}).
\end{equation*}
Since $g\tan\norm{g}_{B}\in \mathcal{V}_{m}$, we can rewrite the local optimization problem into an eigenvalue problem as 
\begin{equation*}
	\min_{g\in\mathcal{V}_{m}}\Rq\bigl(\exp_{y_{m}}(g)\bigr) = \min_{g\in\mathcal{V}_{m}} \Rq(y_{m}+g) = \min_{x\in\spa{g_{m},x_{m},y_{m}}} \Rq(x), 
\end{equation*}
where we use the homogeneity of $\Rq$ again and $\spa{\log_{y_{m}}(x_{m}),y_{m}}=\spa{x_{m},y_{m}}$. Suppose the minimizer is $z_{m}=c_{1}x_{m}+c_{2}g_{m}+c_{3}y_{m}$, we can update $x_{m+1}$ and $\widetilde{x}_{m+1}$ by 
\begin{equation}
	\label{RAPx}
	\widetilde{z}_{m} = c_{1}\widetilde{x}_{m}+c_{2}\widetilde{g}_{m}+c_{3}\widetilde{y}_{m},\quad 
	x_{m+1} = \frac{z_{m}}{\sqrt{z_{m}^{\Ttran}\widetilde{z}_{m}}},\quad \widetilde{x}_{m+1} = \frac{\widetilde{z}_{m}}{\sqrt{z_{m}^{\Ttran}\widetilde{z}_{m}}}
\end{equation}

Now we propose a Riemannian Acceleration with Preconditioning (RAP) for symmetric eigenvalues problems as \cref{algoPLORAG}. 
Note that the convexity radius of the sphere is $\pi/2$ \cite[Page 84]{Klingenberg1995}, \ie any geodesic ball in $\mathcal{X}_{B}$ is convex. Combining \cref{thmLORAG,rmkcoef,lemcvxfunP}, we can obtain the convergence of \cref{algoPLORAG}.
\begin{theorem}
	\label{thmPLORAG}
	Let $\Rq(x)$ be the Rayleigh quotient defined in \cref{defRQ},
	suppose $\lambda_{1}\leq \rho<(\lambda_{1}+\lambda_{2})/2$ and the parameters $\mu_{B}$ and $L_{B}$ defined in \cref{defmuLP}
	satisfy $L_{B}\geq 9\mu_{B}>0$. Denote $\kappa_{B}=L_{B}/\mu_{B}$.
	When the initial point $x_{0}$ satisfies
	\begin{equation}
		\label{conX0}
		\Rq(x_{0})-\lambda_{1}\leq  \frac{(\rho-\lambda_{1})}{11(1+2\sqrt{\kappa_{B}})^{2}\kappa_{B}},
	\end{equation}
	\cref{algoPLORAG} is monotonic convergence. The rate of convergence is
	\begin{equation}
		\label{asprateEVP}
		\Rq(x_{m})-\lambda_{1}\leq 2\Bigl(1-\frac{1}{2\sqrt{\kappa_{B}}}\Bigr)^{m}\bigl(\Rq(x_{0})-\lambda_{1}\bigr).
	\end{equation}
\end{theorem}
\begin{proof}
	See \cref{appthmPLORAG}.
\end{proof}
\begin{algorithm2e}[htbp]
	\caption{Riemannian Acceleration with Preconditioning (RAP)}
	\label{algoPLORAG}
	\KwIn{Initial $\widetilde{x}_{0}$, matrices $(A,M)$, preconditioner $B$, parameters $\mu$ and $L$.}
	Set $\kappa=\frac{L}{\mu}$, $\beta=\frac{3}{2\sqrt{\kappa}-4}$,  $\alpha=\frac{\sqrt{\beta^{2}+4(1+\beta)\kappa^{-1}}-\beta}{2}$ and  $\gamma=\frac{\alpha\mu}{\alpha+\beta}$\;
	Compute $x_{0}=B^{-1}\widetilde{x}_{0}$ and normalize them by $x_{0}=\frac{x_{0}}{\sqrt{x_{0}^{\Ttran}\widetilde{x}_{0}}}$ and $\widetilde{x}_{0}=\frac{\widetilde{x}_{0}}{\sqrt{x_{0}^{\Ttran}\widetilde{x}_{0}}}$\;
	Set $\widetilde{v}_{0}=\widetilde{x}_{0}$ and $v_{0}=x_{0}$\;
	\For{$m=0,1,2,\dotsc,$}{
		Update $y_{m}$ and $\widetilde{y}_{m}$ by \cref{RAPy}\;
		Update $v_{m+1}$ and $\widetilde{v}_{m+1}$ by \cref{RAPv}\;
		Update $x_{m+1}$ and $\widetilde{x}_{m+1}$ by \cref{RAPx}\;
	}
\end{algorithm2e}

\section{Second order symmetric elliptic eigenvalue problems with Schwarz preconditioners}
In this part, let us turn to eigensolvers based on Schwarz preconditioners \cite[Section~2]{Toselli2005}. We provide an illustration of why $\vartheta$ serves as a superior measurement for preconditioned eigenvalue problems. Recently, Schwarz preconditioners, especially the Domain Decomposition Method (DDM), becomes a popular technique for elliptic eigenvalue problems \cite{Wang2018,Shao2023,Chen2022}, and in our previous work \cite{Chen2022}, an acceleration by involving momentum terms is observed in numerical examples. In this section, we aim to offer a theoretical explanation for this observed phenomenon.

Our discussion in this part will stand on finite elements methods viewpoints, where we are more interested in the relationship with mesh sizes $h$ and $H$.
To maintain conciseness and avoid an excessive proliferation of constants, we introduce the notation $\alpha \lesssim \beta$ to represent the statement that $\alpha\leq c\beta$, where the constant $c$ is positive and independent of mesh sizes $h$ and $H$, and the relative spectral gap as $\gap=\lambda_{2}/(\lambda_{2}-\lambda_{1})$.

\subsection{Theory of Schwarz and domain decomposition methods}
Suppose $\Omega \subset \R^{d}$ for $d=2$ or $3$ is a convex polygonal domain, the matrix $\{a_{ij}(x)\}_{i,\,j=1}^{d}$ is symmetric and uniformly positive definite, and $a_{ij}(x) \in C^{0,1}(\overline{\Omega})$ for $i, j = 1, \dotsc, d$. Let $V_{H} \subset V_{h} \subset H_{0}^{1}(\Omega)$ be the continuous and piecewise linear element spaces based on quasi--uniform triangular partitions $\mathcal{T}_{H}$ and $\mathcal{T}_{h}$, where $\mathcal{T}_{h}$ is a refinement of $\mathcal{T}_{H}$, and $0 < h < H < 1$ are the mesh sizes of $\mathcal{T}_{h}$ and $\mathcal{T}_{H}$, respectively.
We consider the following second-order symmetric elliptic eigenvalue problem:
\begin{equation}
    \label{lapevpfem}
    a(u_{1}, v) = \lambda_{1}\dual{u_{1}, v} \quad \forall\, v \in V_{h}, \quad \norm{u_{1}} = 1 \quad \text{and} \quad u_{1} \in V_{h},
\end{equation}
where $\dual{\cdot, \cdot}$ and $\norm{\cdot}$ are the $L^{2}$ inner-product in $\Omega$ and its associated norm, and
\begin{equation}
    a(u, v) \defi  \sum_{i, j = 1}^{d} \int_{\Omega} a_{ij}(x) \frac{\partial u}{\partial x_{i}} \frac{\partial v}{\partial x_{j}} \de x.
\end{equation}
For any $u \in V_{h}$, let $A^{-1}$ be the global solver by defining $A^{-1}u \in V_{h}$ as 
\begin{equation*}
    a(A^{-1}u, v) = \dual{u, v} \quad \forall\, v \in V_{h}.
\end{equation*}
We are interested in additive Schwarz preconditioners \cite[Section~2]{Toselli2005} for $A^{-1}$, denoted by $B^{-1}$.
Consider a family of subspaces of $V_{h}$ as $\{V_{i}\}_{i=0}^{N}$, assume $V_{h}$ admits the following decomposition:
\begin{equation}
    V_{h}=R_{0}^{\Ttran}V_{0}+\sum_{i=1}^{N}R_{i}^{\Ttran}V_{i},
\end{equation}
where $R_{i}^{\Ttran}\colon V_{i}\mapsto V$ are the interpolation operators for $i=0,\dotsc,d$.
The additive Schwarz preconditioner can be defined as
\begin{equation}
\label{defddmB}
	B^{-1} \defi \sum_{i=0}^{N}R_{i}^{\Ttran}B_{i}^{-1}R_{i},
\end{equation}
where $B_{i}^{-1}$ is the solver on $V_{i}$, and $R_{i}$ is the dual operator of $R_{i}^{\Ttran}$ in $\dual{\cdot,\cdot}$.
To aid in understanding, we present a specific example of additive Schwarz preconditioners, following the structure outlined in \cite[Section~7.4]{Brenner2008}.
This example, two-level overlapping domain decomposition, is also the preconditioner we use in numerical experiments.
\begin{example}[Two-level overlapping domain decomposition preconditioner]
\label{exp_dd}
    Consider a two-dimensional region $\Omega = [0, 1]^2$. Let $\mathcal{T}_{H}$ be a coarse triangulation as shown in \cref{figDD}. The region $\Omega$ is divided into non-overlapping subdomains $\tilde{\Omega}_{j}$ for $1 \leq j \leq 16$, which are aligned with $\mathcal{T}_{H}$. Subsequently, $\mathcal{T}_{H}$ is further subdivided to obtain a finer triangulation $\mathcal{T}_{h}$.

    Define $\Omega_{j} = \tilde{\Omega}_{j, \delta} \cap \overline{\Omega}$, where $\tilde{\Omega}_{j, \delta}$ is an open set obtained by enlarging $\tilde{\Omega}_{j}$ by a band of width $\delta$, ensuring $\Omega_{j}$ is aligned with $\mathcal{T}_{h}$ as shown in \cref{figDD}. Typically, we assume the overlapping ratio $\delta/H$ is bounded below by a constant, which is $0.5$ in this case.

	Let $V_{0}=V_{H}$ be the coarse space, and $V_{j}\subset V_{h}$ denote the space containing functions supported in $\Omega_{j}$ for $1\leq j\leq 16$. 
    Define $B_{0}^{-1}$ and $B_{j}^{-1}$ as a coarse solver and some local solvers as 
    \begin{equation*}
        \begin{aligned}
            a(B_{0}^{-1}u_{H},v_{H})&=\dual{u_{H},v_{H}}\quad \forall\, v_{H} \in V_{H},\\ 
            a(B_{j}^{-1}u_{j},v_{j})&=\dual{u_{j},v_{j}}\quad \forall\, v_{j} \in V_{j}.
        \end{aligned}
    \end{equation*}
    The two-level overlapping domain decomposition preconditioner is given by \cref{defddmB}, where $R_{i}^{\Ttran}$ is just a natural injection operator, \ie $R_{i}^{\Ttran}v_{j}=v_{j}$ for all $v_{j}\in V_{j}$.
    \begin{figure}[htbp]
        \centering
        \includegraphics[width=0.7\textwidth]{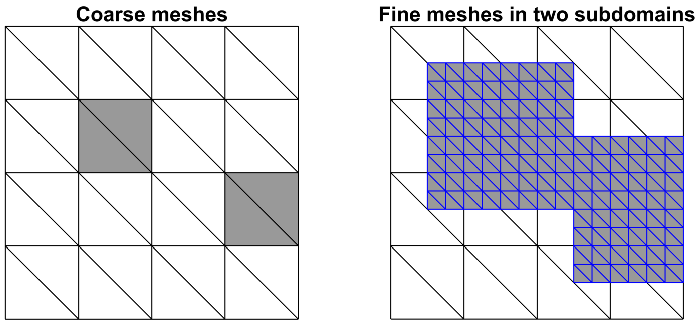}
        \caption{Construction of an overlapping domain decomposition.}
        \label{figDD}
    \end{figure}
\end{example}

For further analysis of additive Schwarz preconditioners, we make some classical assumptions on that.
\begin{Assumption}[{\cite[Assumptions~2.2--2.4]{Toselli2005}}]$\ $
	\label{aspAS}
	\begin{itemize}
		\item Stable decomposition: There exists a constant $C_{s}$ such that every $v\in V_{h}$ admits a decomposition $v = \sum_{i=0}^{N}R_{i}^{\Ttran}v_{i}$
		      that satisfies $\sum_{i=1}^{N}\norm{v_{i}}_{B_{i}}^{2}\leq C_{s}^{2}\norm{v}_{A}^{2}$.
		\item Strengthened Cauchy-Schwarz inequalities: There exist constants $0\leq \epsilon_{ij}\leq 1$, $1\leq i,j\leq N$, such that $(R_{i}^{\Ttran}v_{i},R_{j}^{\Ttran}v_{j})_{A}\leq \epsilon_{ij}\norm{R_{i}^{\Ttran}v_{i}}_{A}\norm{R_{j}^{\Ttran}v_{j}}_{A}$ for $v_{i}\in V_{i}$ and $v_{j}\in V_{j}$.
		\item Finite covering: Assume there are at most $N_{0}$ nonzero elements at each column of $\mathcal{E} = \{\epsilon_{ij}\}_{i,j=1}^{N}$, where $N_{0}$ is independent of $H$ and $h$. Then the spectral radius of $\mathcal{E}$ is bounded by $N_{0}$.
		\item Local stability: There exists $\omega>0$, such that $\norm{R_{i}^{\Ttran}v_{i}}_{A}^{2}\leq \omega \norm{v_{i}}_{B_{i}}^{2}$ holds for all $v_{i}\in V_{i}$ and $0\leq i\leq N$.
	\end{itemize}
\end{Assumption}

With \cref{aspAS}, we can obtain the following results about additive Schwarz preconditioners and finite elements method, whose proof can be found in \cite{Toselli2005,Babuska1989,Brenner2008}.
\begin{proposition}
	\label{kappaAS}
	Under \cref{aspAS}, the smallest and largest eigenvalues of $B^{-1}A$, \ie $\nu_{\min}$ and $\nu_{\max}$, and the condition number $\kappa_{\nu}$ satisfy
	\begin{equation}
		\nu_{\min}\geq 1,\quad \nu_{\max} \leq C_{s}^{2}\omega( N_{0}+1),\quad  \kappa_{\nu}\leq C_{s}^{2}\omega( N_{0}+1).
	\end{equation}
	Specifically, in DDM, $N_{0}\approx 1$, $\omega\approx 1$, and $C_{s}\lesssim 1$ for the two-level overlapping method \cite{Dryja1989} and $C_{s}\lesssim 1+\ln(H/h)$ for the Bramble-Pasciak-Schatz (BPS) method \cite{Bramble1986}.
\end{proposition}

\begin{proposition}
	\label{propAS}
	Let $(\lambda_{0},u_{0})$ be the smallest eigenvalue and its corresponding eigenvector of $B_{0}$, \ie
	\begin{equation*}
		(B_{0}u_{0},v_{0})=\lambda_{0}\dual{u_{0},v_{0}}\quad\forall\, v_{0}\in V_{0},\quad\text{and}\quad \norm{u_{0}}=1.
	\end{equation*}
	Define $C_{c}$, $C_{P}$ and $C_{a}$ as follows:
	\begin{equation*}
		\begin{aligned}
			C_{c}&\defi \min_{\alpha\in\R}\norm{u_{1}-\alpha R_{0}^{\Ttran}u_{0}}_{A},\\ 
			C_{P}&\defi \max_{1\leq i\leq N}\lambda_{\max}(B_{i}^{-1})
		      =\max_{1\leq i\leq N}\lambda_{\min}^{-1}(B_{i}),\\ 
			C_{a} &\defi \max_{v\in V}\frac{\sum_{i=1}^{N}\norm{R_{i}v}^{2}}{\norm{v}^{2}},
		\end{aligned}
	\end{equation*}
	Under \cref{aspAS}, when $H$ is sufficiently small, it holds that 
	\begin{equation*}
		C_{c}\lesssim \lambda_{1}^{1/2}H,\quad 0\leq \lambda_{0}-\lambda_{1}\lesssim \lambda_{1}H^{2},\quad C_{P}\lesssim \lambda_{1}^{-1}H^{2}\quad\text{and}\quad C_{a}\lesssim 1.
	\end{equation*}
\end{proposition}

\subsection{Analysis for leading angle}
Now let us consider the leading angle defined in \cref{defvartheta} in the Schwarz framework.
In order to use \cref{estLA}, we first consider the global term $R_{0}^{\Ttran}B_{0}^{-1}R_{0}u_{1}-\alpha u_{1}$ and the local term $\sum_{i=1}^{N}R_{i}^{\Ttran}B_{i}^{-1}R_{i}u_{1}$, respectively.

\begin{lemma}
	\label{estCoarse}
	Let $u_{1}$ be an eigenvector of \cref{lapevpfem}, under \cref{aspAS},
	\begin{equation}
		\norm{R_{0}^{\Ttran}B_{0}^{-1}R_{0}u_{1}-\lambda_{0}^{-1}u_{1}}_{A}\leq (\omega\lambda_{0}^{-1/2}\lambda_{1}^{-1/2}+\lambda_{0}^{-1})C_{c}\lesssim \lambda_{1}^{-1/2}H.
	\end{equation}
\end{lemma}
\begin{proof}
	For any $\alpha\in\R$, by the triangular inequality and local stability in \cref{aspAS},
	\begin{equation*}
		\begin{aligned}
			\norm{R_{0}^{\Ttran}B_{0}^{-1}R_{0}u_{1}-\lambda_{0}^{-1}u_{1}}_{A}
			 & \leq \norm{R_{0}^{\Ttran}B_{0}^{-1}(R_{0}u_{1}-\alpha u_{0})}_{A}+\lambda_{0}^{-1}\norm{\alpha R_{0}^{\Ttran}u_{0}-u_{1}}_{A} \\
			 & \leq \omega \norm{R_{0}u_{1}-\alpha u_{0}}_{B_{0}^{-1}}+\lambda_{0}^{-1}\norm{\alpha R_{0}^{\Ttran}u_{0}-u_{1}}_{A}           \\
			 & \leq \omega\lambda_{0}^{-1/2}\norm{\alpha u_{0}-R_{0}u_{1}} +\lambda_{0}^{-1}\norm{\alpha R_{0}^{\Ttran}u_{0}-u_{1}}_{A}.
		\end{aligned}
	\end{equation*}
	Note that for any $v_{0}$ and $w_{0}$ in $V_{0}$,
	\begin{equation*}
		\dual{v_{0},w_{0}}=\dual{R_{0}^{\Ttran}v_{0},R_{0}^{\Ttran}w_{0}} =\dual{v_{0},R_{0}R_{0}^{\Ttran}w_{0}}.
	\end{equation*}
	We know $R_{0}R_{0}^{\Ttran}$ is the identity operator on $V_{0}$. For any $v$ in $V_{h}$, since
	\begin{equation*}
		\norm{R_{0}v} = \sup_{v_{0}\in V_{0}}\frac{\dual{R_{0}v,v_{0}}}{\norm{v_{0}}} = \sup_{v_{0}\in V_{0}}\frac{\dual{v,R_{0}^{\Ttran}v_{0}}}{\norm{v_{0}}} \leq \norm{v}\sup_{v_{0}\in V_{0}}\frac{\norm{R_{0}^{\Ttran}v_{0}}}{\norm{v_{0}}} = \norm{v},
	\end{equation*}
	we have
	\begin{equation*}
		\norm{\alpha u_{0}-R_{0}u_{1}}=\norm{\alpha R_{0}R_{0}^{\Ttran}u_{0}-R_{0}u_{1}}\leq \norm{\alpha R_{0}^{\Ttran}u_{0}-u_{1}}\leq \lambda_{1}^{-1/2}\norm{\alpha R_{0}^{\Ttran}u_{0}-u_{1}}_{A}.
	\end{equation*}
	Then the lemma is proved by $C_{c}$ in \cref{propAS}.
\end{proof}

\begin{lemma}
	\label{estLoc}
	Under \cref{aspAS}, let
	\begin{equation}
		B_{L}^{-1}=\sum_{i=1}^{N}R_{i}^{\Ttran}B_{i}^{-1}R_{i},
	\end{equation}
	then for any $u\in V_{h}$ with $\norm{u}=1$,
	\begin{equation}
		\norm{B_{L}^{-1}u}_{A}^{2}\leq \omega N_{0} C_{a}C_{P}\lesssim \lambda_{1}^{-1}H^{2}.
	\end{equation}
\end{lemma}
\begin{proof}
	By the local stability in \cref{aspAS} and $C_{P}$ in \cref{propAS},
	\begin{equation}
		\label{estLoc1}
			\norm{R_{i}^{\Ttran}B_{i}^{-1}R_{i}u}_{A}^{2}
			\leq \omega \norm{R_{i}u}_{B_{i}^{-1}}^{2}
			\leq \omega C_{P} \norm{R_{i}u}^{2}.	
	\end{equation}
	Then the lemma is proved by 
	\begin{equation*}
		\norm{B_{L}^{-1}u}_{A}^{2}
		\leq  N_{0} \sum_{i=1}^{N}\norm{R_{i}^{\Ttran}B_{i}^{-1}R_{i}u}_{A}^{2}
		\leq \omega N_{0} C_{P}  \sum_{i=1}^{N}\norm{R_{i}u}^{2}
		\leq \omega N_{0} C_{a}C_{P},
	\end{equation*}
	where we use the finite covering in \cref{aspAS} and $C_{a}$ in \cref{propAS}.
\end{proof}

Combining \cref{estCoarse,estLoc}, we obtain the bound for $\vartheta$ defined \cref{defvartheta}.

\begin{theorem}
	\label{thmAS}
	With \cref{aspAS}, the leading angle $\vartheta$ in \cref{estLA} can be bounded by
	\begin{equation*}
		\cos\vartheta \lesssim \sqrt{\kappa_{\nu}}\Bigl(\frac{\lambda_{1}^{-1}-\rho^{-1}}{\lambda_{1}^{-1}-\lambda_{2}^{-1}}\Bigr)^{1/2}+H,
	\end{equation*}
	where the condition number $\kappa_{\nu}$ can be bounded by
	\begin{equation}
		\label{kappaDDM}
		\kappa_{\nu} \lesssim \left\{
		\begin{aligned}
			 & 1\quad                          & \text{for two-level overlapping method,} \\
			 & \bigl(1+\ln(H/h)\bigr)^{2}\quad & \text{for BPS method.}
		\end{aligned}
		\right.
	\end{equation}
\end{theorem}
\begin{proof}
	By \cref{thmvartheta}, it is sufficient to show that 
	\begin{equation*}
		\min_{\alpha\in\R }\frac{\norm{B^{-1}Au_{1}-\alpha u_{1}}_{A}}{\sqrt{\lambda_{1}}}\lesssim H.
	\end{equation*}
	Taking $\alpha = \lambda_{1}\lambda_{0}^{-1}$ and using the triangular inequality, we know 
	\begin{equation*}
		\frac{\norm{B^{-1}Au_{1}-\lambda_{1}\lambda_{0}^{-1} u_{1}}_{A}}{\lambda_{1}} = \norm{B^{-1}u_{1}-\lambda_{0}^{-1}u_{1}}_{A}\leq \norm{B_{0}^{-1}u_{1}-\lambda_{0}^{-1}u_{1}}_{A}+\norm{B_{L}^{-1}u_{1}}_{A}.
	\end{equation*}
	Then the theorem is proved by \cref{estCoarse,estLoc}.
\end{proof}

\subsection{Analysis for convergence rate}
In \cref{algoPLORAG}, the initial point $x_{0}$ is computed as $x_{0}=B^{-1}\widetilde{x}_{0}$, where $\widetilde{x}_{0}$ is the input. The following lemma gives an estimation about $\Rq(x_{0})$.
\begin{lemma}
	\label{lemx0}
	Assume the $\widetilde{x}_{0}$ is the eigenvector calculated on the coarse space $V_{0}$, \ie $\widetilde{x}_{0}=R_{0}^{\Ttran}u_{0}$, let $x_{0}=B^{-1}\widetilde{x}_{0}$, then,
	\begin{equation*}
		\Rq(x_{0})-\lambda_{1}\lesssim \lambda_{1}H^{2}.
	\end{equation*}
\end{lemma}
\begin{proof}
	See \cref{applemx0}.
\end{proof}

\begin{theorem}
	Assume the initial point $\widetilde{x}_{0}$ is the eigenvector calculated on the coarse space $V_{0}$, \ie $\widetilde{x}_{0}=R_{0}^{\Ttran}u_{0}$. When $H$ is sufficiently small and satisfying
	\begin{equation}
		\label{Hsuffsmall}
		H\leq \frac{\gap^{3}-\gap^{4}}{\kappa_{\nu}^{4}}, \quad\text{where}\quad \gap = \frac{\lambda_{2}-\lambda_{1}}{\lambda_{2}},
	\end{equation}
	then \cref{algoPLORAG} is monotonic convergence, and the convergence rate is
	\begin{equation*}
		\Rq(x_{m})-\lambda_{1}\leq 2\bigl(1-C(\gap/\kappa_{\nu})^{1/2}\bigr)^{m}\bigl(\Rq(x_{0})-\lambda_{1}\bigr),
	\end{equation*}
	where $C$ is a constant independent of $h$, $H$ and $(\lambda_{2}-\lambda_{1})/\lambda_{2}$.
\end{theorem}
\begin{proof}
	First, let us pick a $\rho$ as
	\begin{equation*}
		\rho=\lambda_{1}+H\kappa_{\nu}^{2}\lambda_{2}\leq \lambda_{1}+\frac{\gap^{2}-\gap^{3}}{\kappa_{\nu}^{2}}(\lambda_{2}-\lambda_{1})< \frac{\lambda_{1}+\lambda_{2}}{2}.
	\end{equation*}
	Then we know that 
	\begin{equation*}
		\varepsilon = \frac{\lambda_{1}^{-1}-\rho^{-1}}{\lambda_{1}^{-1}-\lambda_{2}^{-1}} = \frac{\lambda_{2}(\rho-\lambda_{1})}{\rho(\lambda_{2}-\lambda_{1})} \leq \frac{H\kappa_{\nu}^{2}}{\gap-\gap^{2}}.
	\end{equation*}
	Note that in \cref{thmAS}, we have derived that 
	\begin{equation*}
		\cos\vartheta \lesssim \sqrt{\kappa_{\nu}\varepsilon}+H.
	\end{equation*}
	By the result for condition number in \cref{rmkPEVP}, we know that
	\begin{equation*}
		\kappa_{B} = \frac{L_{B}}{\mu_{B}} 
			\leq \frac{\kappa_{\nu}}{\gap}\Bigl(1+\order\bigl(\kappa_{\nu}(\gap^{-2}\cos^{2}\vartheta+\varepsilon)\bigr)\Bigr)
			\leq \frac{\kappa_{\nu}}{\gap}\Bigl(1+\order\bigl(\frac{H\kappa_{\nu}^{4}}{\gap^{3}-\gap^{4}}\bigr)\Bigr)\lesssim \frac{\kappa_{\nu}}{\gap},
	\end{equation*}
	where \cref{Hsuffsmall} is used in the last inequality.	In this case, by \cref{lemx0},
	\begin{equation*}
		\Rq(x_{0})-\lambda_{1}\lesssim \lambda_{1}H^{2}= \frac{(\rho-\lambda_{1})\lambda_{1}H}{\kappa_{\nu}^{2}\lambda_{2}}= C_{0}H \frac{(\rho-\lambda_{1})}{11(1+2\sqrt{\kappa_{B}})^{2}\kappa_{B}},
	\end{equation*}
	where 
	\begin{equation*}
		C_{0} = \frac{11\lambda_{1}(1+2\sqrt{\kappa_{B}})^{2}\kappa_{B}}{\lambda_{2}\kappa_{\nu}^{2}}\lesssim \frac{1-\gap}{\gap^{2}}.
	\end{equation*}
	Thus, when $H$ is small enough, all assumptions in \cref{thmPLORAG} hold. Then the theorem is obtained by \cref{thmPLORAG}.
\end{proof}
\begin{remark}
	When applying the two-level overlapping or BPS domain decomposition preconditioner, the requirement \cref{Hsuffsmall} is a mild condition due to \cref{kappaDDM}.
\end{remark}
\begin{remark}
	Compared to the Preconditioned Steepest Descent (PSD) method, the rate of convergence is improved
	\begin{equation*}
		\text{from}\quad 1-C_{\text{PSD}}\frac{\lambda_{2}-\lambda_{1}}{\kappa_{\nu}\lambda_{2}}\quad \text{to}\quad 1-C_{\text{RAP}}\Bigl(\frac{\lambda_{2}-\lambda_{1}}{\kappa_{\nu}\lambda_{2}}\Bigr)^{1/2}.
	\end{equation*}
	Such an improvement significantly accelerates the convergence when $\kappa_{\nu}$ is large or the relative spectral gap $(\lambda_{2}-\lambda_{1})/\lambda_{2}$ is small.
\end{remark}

\section{Numerical experiments}
\label{sec:numexp}
In this section, we compare the following eigensolvers\footnote{Scripts to reproduce numerical results are publicly available at \url{https://github.com/nShao678/RAP-code}}:
\begin{itemize}
	\item RAP: Riemannian acceleration with preconditioning,
	\item LOPCG: locally optimal preconditioned conjugate gradient \cite{Knyazev2001},
	\item PSD: preconditioned steepest descent \cite{Neymeyr2012},
	\item Lanczos: generalized Lanczos method without restarting \cite[Algorithm~9.1]{Saad2011}.
\end{itemize}

Our objective is to observe and analyze the acceleration effect achieved through preconditioning and the inclusion of a momentum term when solving the 2D discrete Laplacian eigenvalue problem:
\begin{equation*}
	-\Delta^{h} u_{1} = \lambda_{1} u_{1}
\end{equation*}
within the unit square $\Omega=(0,1)^{2}$ , subject to the Dirichlet boundary condition $u_{1}=0$ on $\partial\Omega$.
The standard finite element method is used for discretization, and a two-level overlapping domain decomposition method, as shown in \cref{exp_dd}, is employed as the preconditioner. The coarse mesh size is set to $H=2^{-2}$, with an overlapping ratio of $2$. A range of fine mesh sizes, from $2^{-3}$ to $2^{-10}$, is considered. 
For the RAP method, the parameters $\mu_{B}=8$ and $L_{B}=50$ are used.
The stopping criterion is $\rho-\lambda_{1} \leq 10^{-6}\lambda_{1}$,
where $\rho$ is the computed Rayleigh quotient, and $\lambda_{1}$ is the smallest eigenvalue of the discrete Laplacian operator.

\begin{table}[htbp]
	\centering
	\caption{Iteration numbers Laplacian eigenvalue problems.}
	\begin{tabular}{|c|c|c|c|c|c|c|c|c|}
		\hline
		$h$   & $2^{-3}$ & $2^{-4}$ & $2^{-5}$ & $2^{-6}$ & $2^{-7}$ & $2^{-8}$ & $2^{-9}$ & $2^{-10}$ \\ \hline
		RAP & $8$     & $8$     & $8$     & $8$     & $8$     & $8$     & $8$     & $8$      \\ \hline
		LOPCG & $8$     & $8$     & $8$     & $8$     & $8$     & $8$     & $8$     & $8$      \\ \hline
		PSD   & $14$     & $14$     & $14$     & $14$     & $14$     & $14$     & $14$     & $14$         \\ \hline
		Lanczos  & $12$     & $24$     & $48$    & $97$    & $193$    & $386$    & $766$   & $1438$    \\ \hline
	\end{tabular}
	\label{tableDDM}
\end{table}

As can be observed in the iteration numbers highlighted in \cref{tableDDM}, several noteworthy observations can be made. In the absence of preconditioning, even the Lanczos method, which is frequently regarded as the most efficient first-order method in terms of iterations (matrix-vector multiplications), exhibits a discernible increase in the number of iterations, growing at a rate of $\order(h^{-1})$ as $h\to 0$. It is evident that preconditioning becomes a crucial aspect in the context of handling large-scale elliptic problems. Methods that incorporate a two-level overlapping domain decomposition preconditioner, such as LOPCG, RAP, and PSD, demonstrate the ability to converge with a relatively limited number of iterations. However, LOPCG and RAP require approximately half the iterations compared to the PSD method. Given that the computational cost per iteration is comparable, momentum-based methods like LOPCG and RAP present a more efficient alternative.

\section{Concluding remarks}
In this paper, we introduced new tools for analyzing eigenvalue problems with preconditioning on Riemannian manifolds and proposed a Riemannian acceleration method with preconditioning. The acceleration of this method is both theoretically proven and numerically validated. Future research will focus on extending this theory to block versions.
\section*{Statements and Declarations}
\begin{itemize}
	\item Acknowledgments: 
	We extend our heartfelt thanks to Zhaojun Bai and Daniel Kressner for their invaluable help. Additionally, we express our gratitude to Foivos Alimisis, Long Chen, Yuhang Liu, Meiyue Shao, Bart Vandereycken, and Xuejun Xu for their insightful discussions, which significantly enriched this work. We would also like to thank anonymous referees for their constructive comments and suggestion of the early version of the manuscript.
	\item Funding: Chen is supported by the National Natural Science Foundation of China (NSFC) 12241101 and 12071090.
	\item Competing interests: Parts of this work were carried out when Shao was at School of Mathematical Scineces, Fudan University. The authors declare to have no competing interests related to this work.
\end{itemize}

\begin{appendices}

	\crefalias{section}{appendix}
	\crefalias{subsection}{appendix}
	\section{Instruments on Riemannian manifolds}
	\label{appRO}
	First, we present some results about sphere, the Riemannian manifold we focus on. These results mainly come from \cite{Boumal2023}.
	\begin{proposition}
		\label{propRM}
		Let $\Sm$ be the unit sphere in $\R^{n}$, with the Riemannian metric induced by the standard Euclidean inner product, the following properties hold.
		\begin{itemize}
			\item For any $x\in\Sm$, the tangent space $T_{x}\Sm$ is
				  \begin{equation*}
					  T_{x}\Sm\defi\{v\in\R^{n}\mid x^{\Ttran}v=0,\,x\in\Sm\}.
				  \end{equation*}
				  And the projection of $\xi\in\R^{n}$ to $T_{x}\Sm$ is
				  \begin{equation*}
					  P_{x}\xi=(I-xx^{\Ttran})\xi.
				  \end{equation*}
			\item The geodesic is a curve $\gamma\colon [0,1]\mapsto \mathcal{M}$ of constant speed and locally minimum length, which plays the role of straight line in Riemannian manifolds. For any $x\in\Sm$ and $v\in T_{x}\Sm$, the geodesic $\gamma$ with initial value $x$ and speed $v$ is
				  \begin{equation*}
					  \gamma(t) = \cos(\norm{v}t)x+\sin(\norm{v}t)\frac{v}{\norm{v}}.
				  \end{equation*}
			\item Let $\exp_{x}\colon T_{x}\mathcal{M}\mapsto \mathcal{M}$ be the exponential map, which is defined as $\exp_{x}(v)=\gamma(1)$, where $\gamma$ is the unique geodesic satisfying $\gamma(0)=x$ and $\gamma^{\prime}(0)=v$. The exponential map on sphere is
				  \begin{equation*}
					  \exp_{x}(v) = \gamma(1) = \cos(\norm{v})x+\sin(\norm{v})\frac{v}{\norm{v}}.
				  \end{equation*}
			\item For a local neighborhood $\mathcal{U}$ of $x$, the exponential map $\exp_{x}$ is a diffeomorphism in $\mathcal{U}$, and its inverse, which is called logarithmic map, can be defined as $\log_{x}=\exp_{x}^{-1}\colon \mathcal{U}\mapsto T_{x}\mathcal{M}$. For $x\neq y\in\Sm$, the logarithmic map is
				  \begin{equation*}
					  \log_{x}(y) = \arccos(x^{\Ttran}y)\frac{P_{x}y}{\norm{P_{x}y}}.
				  \end{equation*}
				  Moreover, the norm $\norm{\log_{x}(y)}$ is the distance between $x$ and $y$ on $\mathcal{M}$.
			\item For any $x$ and $y$ in $\mathcal{M}$, let $\Gamma_{x}^{y}\colon T_{x}\mathcal{M}\mapsto T_{y}\mathcal{M}$ be the parallel transport, which transports tangent vector from $T_{x}\mathcal{M}$ to $T_{y}\mathcal{M}$ along the geodesic $\gamma$, where $\gamma(0)=x$ and $\gamma(1) = y$. On sphere, let $v=\log_{x}(y)$, then
				  \begin{equation*}
					  \Gamma_{x}^{y}(u) = \Bigl(I+\bigl(\cos(t\norm{v})-1\bigr)\frac{vv^{\Ttran}}{\norm{v}^{2}}-\sin(t\norm{v})\frac{xv^{\Ttran}}{\norm{v}}\Bigr)u.
				  \end{equation*}
		\end{itemize}
	\end{proposition}
	
	Apart from these basic properties in \cref{propRM}, we also need to handle the distortion of Riemannian manifolds. Suppose $U$ is a $2$-dimensional plane of $T_{x}\mathcal{M}$, the sectional curvature of $\mathcal{M}$ at $x$ associated with the plane $U$, which is denoted as $\kappa(x,U)$, is the Gaussian curvature of $U$ at $x$. 
	
	Since we are studying convex optimization problems, some background in geodesic convexity and gradients is also necessary, primarily drawn from \cite{Zhang2018}.
	
	\begin{definition}[Geodesically convex set]
		\label{defGCS}
		Let $\mathcal{X}\subset\mathcal{M}$. If for any $x,\,y\in\mathcal{X}$, there is a geodesic $\gamma$ with $\gamma(0)=x$ and $\gamma(1)=y$ such that $\gamma(t)\in \mathcal{X}$ for any $0\leq t\leq 1$, then $\mathcal{X}$ is called geodesically convex.
	\end{definition}
	\begin{definition}[Geodesically convex function]
		\label{defGCF}
		Let $f\colon \mathcal{X}\subset\mathcal{M}$, where $\mathcal{X}$ is a geodesically convex set. If for any $x,\,y\in\mathcal{X}$, there is a geodesic $\gamma$ with $\gamma(0)=x$ and $\gamma(1)=y$ such that
		\begin{equation*}
			f\bigl(\gamma(t)\bigr)\leq (1-t)f(x)+tf(y)\quad\forall\, 0\leq t\leq1,
		\end{equation*}
		then $f$ is called geodesically convex.
	\end{definition}
	\begin{definition}[Gradient]
		Let $f\colon \mathcal{M}\mapsto \R$ be a differentiable function, then the gradient of $f$ at $x\in\mathcal{M}$, denoted by $\grad f(x)$ is defined as the unique element of $T_{x}\mathcal{M}$ that satisfies
		\begin{equation*}
			\dual{\grad f(x),\xi} = \mathrm{D}f(x)[\xi]\quad\forall\, \xi\in T_{x}\mathcal{M},
		\end{equation*}
		where $\mathrm{D}f(x)[\xi]$ is the directional derivative of $f$ along $\xi$.
	\end{definition}
	\begin{proposition}
		If $\log_{x}$ is well-defined for any $x\in\mathcal{X}$ and $f$ is a differentiable function, an equivalent definition for a geodesically convex function is that: there exists a constant $\mu>0$ such that
		\begin{equation*}
			f(y)\geq f(x)+\dual{\grad f(x),\log_{x}(y)}+\frac{\mu}{2}\norm{\log_{x}(y)}^{2}\quad\forall\, x,\,y\in\mathcal{X}.
		\end{equation*}
		In this case, we call $f$ as $\mu$-geodesically convex.
	\end{proposition}
	\begin{definition}[Lipschitz smooth function]
		\label{defLS}
		Let $f\colon \mathcal{X}\mapsto \R$ be a differentiable function, assume $\log_{x}$ is well-defined for any $x\in\mathcal{X}$. If there exists a constant $L>0$ such that
		\begin{equation*}
			f(y)\leq f(x)+\dual{\grad f(x),\log_{x}(y)}+\frac{L}{2}\norm{\log_{x}(y)}^{2}\quad\forall\, x,\,y\in\mathcal{X},
		\end{equation*}
		then $f$ is called $L$-Lipschitz smooth.
	\end{definition}
	\begin{proposition}
		If $f$ is $L$-Lipschitz smooth on $\mathcal{X}$, then
		\begin{equation*}
			\norm{\grad f(x)-\Gamma_{y}^{x}\grad f(y)}\leq L\norm{\log_{x}(y)}\quad\forall\, x,\,y\in\mathcal{X}.
		\end{equation*}
	\end{proposition}
	\begin{proposition}[Second-order condition]
		\label{AppSOC}
		Let $f\colon \mathcal{X}\mapsto \R$ be a smooth function, assume $\log_{x}$ is well-defined for any $x\in\mathcal{X}$. The function $f$ is $\mu$-geodesically convex and $L$-Lipschitz smooth. if there exist constants $0<\mu<L$ such that for any $v\in T_{x}\mathcal{X}$,
		\begin{equation*}
			\mu\norm{v}^{2}\leq \dual{v,\nabla^{2}f(x)[v]} \leq L\norm{v}^{2}
		\end{equation*}
		holds, where $\nabla^{2}f(x)$ is the Hessian of $f$ at $x$ defined as
		\begin{equation*}
			\nabla^{2}f(x)[v] = \nabla_{v}\grad f(x),
		\end{equation*}
		and $\nabla_{v}$ is the Riemannian connection.
	\end{proposition}
	
	At the end of this part, we provide some expressions for Riemannian Hessian of Rayleigh quotient. 
	\begin{proposition}
		\label{appHess}
		Let $\Rq(x)$ be the Rayleigh quotient defined in \cref{defRQ}, then for any $v\in T_{x}\Sm$ with $\norm{v}=1$,
		\begin{equation*}
			\dual{v,\nabla^{2}\Rq(x)[v]} = \frac{2\norm{v}_{A}^{2}\Rq(x)}{\norm{x}_{A}^{2}}\Bigl(1-\frac{\Rq(x)}{\Rq(v)}\Bigr)+\frac{8(v^{\Ttran}Mx)\cdot v^{\Ttran}\bigl(Ax-\Rq(x)Mx\bigr)}{\norm{x}_{M}^{4}}.
		\end{equation*}
		Specifically, if $M=I$, then the Riemannian Hessian can be simplified as 
		\begin{equation*}
			\dual{v,\nabla^{2}\Rq(x)[v]} = 2\bigl(\Rq(v)-\Rq(x)\bigr).
		\end{equation*}
	\end{proposition}
	\begin{proof}
		By definition and direct computation, for any $v\in T_{x}\Sm$,
		\begin{equation}
			\label{gradRq}
			\begin{aligned}
				\dual{\grad \Rq(x),v} &= \mathrm{D}\Rq(x)[v] = \lim_{\epsilon\to0}\frac{1}{\epsilon}\Bigl(\Rq(x+\epsilon v)-\Rq(x)\Bigr) \\ 
				&= \lim_{\epsilon\to0}\frac{1}{\epsilon}\Bigl(\frac{x^{\Ttran}Ax+2\epsilon x^{\Ttran}Av}{x^{\Ttran}Mx+2\epsilon x^{\Ttran}Mv}-\frac{x^{\Ttran}Ax}{x^{\Ttran}Mx}\Bigr)\\
				&= \lim_{\epsilon\to0}\frac{\Rq(x)}{\epsilon}\Bigl(\bigl(1+2\epsilon\frac{x^{\Ttran}Av}{x^{\Ttran}Ax}\bigr)\bigl(1-2\epsilon\frac{x^{\Ttran}Mv}{x^{\Ttran}Mx}\bigr)-1\Bigr)\\ 
				&= \frac{2v^{\Ttran}\bigl(A-\Rq(x)M\bigr)x}{x^{\Ttran}Mx}. 
			\end{aligned}
		\end{equation}
		For the Hessian, note that $\Sm$ is a Riemannian submanifold of $\R^{n}$, an Euclidean space. By \cite[Corollary~5.16]{Boumal2023},
		\begin{equation*}
			\nabla^{2}\Rq(x)[v] = \nabla_{v}\grad \Rq(x) = P_{x}\mathrm{D} \overline{G}(x)[v],
		\end{equation*}
		where $P_{x}=I-xx^{\Ttran}$ is a projection and 
		\begin{equation*}
			\overline{G}(x) = \frac{2}{x^{\Ttran}Mx}\bigl(Ax-\Rq(x)Mx\bigr) 
		\end{equation*}
		is a smooth extension of $\grad \Rq(x)$ in $\R^{n}$. Note that $P_{x}v=v$ since $v\in T_{x}\Sm$. With some calculation similar to \cref{gradRq},
		\begin{equation*}
			\begin{aligned}
				\dual{v,\nabla^{2}\Rq(x)[v]}&=v^{\Ttran}
				\mathrm{D} \overline{G}(x)[v]
				= \lim_{\epsilon\to0}\frac{1}{\epsilon}v^{\Ttran}\Bigl(\overline{G}(x+\epsilon v)-\overline{G}(x)\Bigr)\\ 
			&=\frac{2}{x^{\Ttran}Mx}\Bigl(v^{\Ttran}\bigl(A-\Rq(x)M\bigr)v-v^{\Ttran}\bigl(Mx\overline{G}^{\Ttran}(x)+\overline{G}(x)(Mx)^{\Ttran}\bigr)v\Bigr)\\ 
			&=J_{1}-J_{2},
			\end{aligned}
		\end{equation*}
		where 
		\begin{equation*}
			J_{1} = 2\frac{v^{\Ttran}(A-\Rq(x)M)v}{x^{\Ttran}Mx} = \frac{2\norm{v}_{A}^{2}\Rq(x)}{\norm{x}_{A}^{2}}\Bigl(1-\frac{\Rq(x)}{\Rq(v)}\Bigr)
		\end{equation*}
		and 
		\begin{equation*}
			\begin{aligned}
				J_{2} &= \frac{2}{x^{\Ttran}Mx}\Bigl(v^{\Ttran}\bigl(Mx\overline{G}^{\Ttran}(x)+\overline{G}(x)(Mx)^{\Ttran}\bigr)v\Bigr)\\ 
				&= \frac{8(x^{\Ttran}Mv)\cdot v^{\Ttran}\bigl(Ax-\Rq(x)Mx\bigr)}{\norm{x}_{M}^{4}}.
			\end{aligned}
		\end{equation*}
	\end{proof}
	
	\section{Proof of some theoretical results}
	\label{appTS}
	\subsection{Proof of \cref{lemcvxset}}
		\label{applemcvxset}
		Since the hemisphere $\Sm_{+}=\{x\in\Sm,\,x^{\Ttran}u_{M}>0\}$ is geodesically convex, it is sufficient to prove for any $x$ and $y$ in $\mathcal{X}_{M}$ and the geodesic $\gamma$ satisfying $\gamma(0)=x$ and $\gamma(1)=y$, the inequality
		\begin{equation}
			\label{cvxset0}
			\norm{\gamma(t)}_{A_{M}}^{2}\leq \rho
		\end{equation}
		holds for all $0\leq t\leq 1$.
		Note that for any $x\in\Sm$, if $x$ is orthogonal to $u_{M}$, then $x^{\Ttran}A_{M}x\geq \lambda_{2}>\rho$, which means $\mathcal{X}_{M}$ is a closed subset of the hemisphere $\Sm_{+}$. If \cref{cvxset0} does not hold for some $x$ and $y$ in $\mathcal{X}_{M}$, since $\mathcal{X}_{M}$ is closed, we may assume
		\begin{equation}
			\label{cvxset}
			x^{\Ttran}A_{M}x=\rho,\quad \gamma(0)=x,\quad\gamma(1)=y,\quad\norm{\gamma(t)}_{A_{M}}^{2}>\rho\quad\text{for all }0<t<1.
		\end{equation}
		Let $v=\log_{x}(y)$ be the logarithmic map and
		\begin{equation*}
			\phi(t)=\norm{\gamma(t)}_{A_{M}}^{2}-\rho=\Bigl(\frac{v^{\Ttran}A_{M}v}{v^{\Ttran}v}-\rho\Bigr)\sin^{2}(\norm{v}t)+\frac{x^{\Ttran}A_{M}v}{\norm{v}}\sin(2\norm{v}t).
		\end{equation*}
		According to \cref{cvxset}, we have
		\begin{equation}
			\label{cvxset2}
			\phi^{\prime}(0)=2x^{\Ttran}A_{M}v \geq 0.
		\end{equation}
		Using \cref{estRqMv}, when $\lambda_{1}\leq \rho< (\lambda_{1}+\lambda_{2})/2$, we obtain
		\begin{equation}
			\label{cvxset3}
			\frac{v^{\Ttran}A_{M}v}{v^{\Ttran}v}\geq \lambda_{1}+\lambda_{2}-\rho\geq\rho\quad\text{and}\quad
			\cos(\norm{v})=x^{\Ttran}y\geq 1-\frac{2(\rho-\lambda_{1})}{\lambda_{2}-\lambda_{1}}> 0,
		\end{equation}
		where we use \cite[Lemma~2.2]{Shao2023a} for the lower bound of $\abs{x^{\Ttran}y}$. 
		Since $\norm{v}$ is the length of geodesic, which is a minor arc, when $x$ and $y$ are different points, we know $0<\norm{v}<\pi/2$. Combining \cref{cvxset2,cvxset3}, we have
		\begin{equation*}
			\begin{aligned}
				0 & \geq \phi(1)=\Bigl(\frac{v^{\Ttran}A_{M}v}{v^{\Ttran}v}-\rho\Bigr)\sin^{2}(\norm{v})+\frac{x^{\Ttran}A_{M}v}{\norm{v}}\sin(2\norm{v})   \geq 0,
			\end{aligned}
		\end{equation*}
		where the equalities hold when
		\begin{equation*}
			\frac{v^{\Ttran}A_{M}v}{v^{\Ttran}v}=\rho
			\quad\text{and}\quad
			x^{\Ttran}A_{M}v=0.
		\end{equation*}
		In this situation, $\phi(t)=0$ for all $t\in[0,1]$, a contradiction to \cref{cvxset}.
	\subsection{Proof of \cref{lemcvxsetP}}
	\label{applemcvxsetP}
		According to \cref{lemcvxset}, we know
		\begin{equation}
			\begin{aligned}
				\mathcal{X} & = \{x^{\Ttran}M_{B}x=1,\, x^{\Ttran}u_{B}>0\mid \frac{x^{\Ttran}A_{B}x}{x^{\Ttran}M_{B}x}\leq\rho\} \\
							& = \{x^{\Ttran}M_{B}x=1,\, x^{\Ttran}u_{B}>0\mid x^{\Ttran}A_{B}x\leq\rho\},
			\end{aligned}
		\end{equation}
		is geodesically convex. Let $\psi\colon \mathcal{X}_{B}\mapsto \mathcal{X}$ be the scaling transformation defined as
		\begin{equation*}
			\psi(x) = \frac{x}{\norm{x}_{M_{B}}}.
		\end{equation*}
		For any $x_{1}$ and $x_{2}$ in $\mathcal{X}_{B}$, and geodesic $\gamma_{B}(t)\in\mathcal{X}_{B}$ satisfying $\gamma_{B}(0)=x_{1}$ and $\gamma_{B}(1)=x_{2}$, there exists a geodesic $\gamma(t)\in\mathcal{X}$ such that $\gamma(0)=\psi(x_{1})$ and $\gamma(1)=\psi(x_{2})$. Note that for any $t\in[0,1]$, there exists a unique $t_{B}\in[0,1]$ such that $\psi\bigl(\gamma(t)\bigr)=\gamma_{B}(t)$.
		Thus, $\mathcal{X}$ is geodesically convex leads to $\mathcal{X}_{B}$ is geodesically convex.
		\subsection{Proof of \cref{eqvAB}}
	\label{appeqvAB}
	
		The direction from right to left is trivial. For the other one, we first show that for any $x\neq\zero$ satisfying $\Rq(x)\leq \rho$, the matrix $[Ax,Bx]$ is rank-one. 
		
		Otherwise, let $V\in\R^{n\times (n-1)}$ be a basis of the $B$-orthogonal complement of $x$. It is clear that $V^{\Ttran}Bx=\zero$ and $V^{\Ttran}Ax=\zero$, where we use $\vartheta=\pi/2$. Then we know the matrix $[V,Ax,Bx]$ is column full rank, which is a contradiction.
	
		Now we know $Au_{1}=\mu Bu_{1}$ since $\Rq(u_{1})=\lambda_{1}<\rho$, and need to show that $A=\mu B$.
		For any $2\leq k\leq n$ and $0< t_{k}^{2}<(\rho-\lambda_{1})/(\lambda_{k}-\rho)$, let $x_{k}=u_{1}+t_{k}u_{k}$. Note that 
		\begin{equation*}
			\Rq(x_{k}) = \frac{\lambda_{1}+t_{k}^{2}\lambda_{k}}{1+t_{k}^{2}}<\rho.
		\end{equation*}
		We know $Ax_{k}=\mu_{k}(t_{k})Bx_{k}$ for some $\mu_{k}(t_{k})>0$.  Since $Au_{1}=\mu Bu_{1}$, substituting $x_{k}=u_{1}+t_{k}u_{k}$ into $Ax_{k}=\mu_{k}(t_{k})Bx_{k}$, we have 
		\begin{equation*}
			\Bigl(1-\frac{\mu_{k}(t_{k})}{\mu}\Bigr)Au_{1}+t_{k}Au_{k}=t_{k}\mu_{k}(t_{k})Bu_{k}.
		\end{equation*} 
		Multiplying $u_{k}^{\Ttran}$ from the left side we find that  
		\begin{equation*}
			\mu_{k}(t_{k}) = \frac{u_{k}^{\Ttran}Au_{k}}{u_{k}^{\Ttran}Bu_{k}}\quad\forall\, 0<t_{k}^{2}<(\rho-\lambda_{1})/(\lambda_{k}-\rho),
		\end{equation*}
		which implies $\mu_{k}(t_{k})$ is a constant independent of $t_{k}$. Note that 
		\begin{equation*}
			\mu_{k}(t_{k}) = \frac{x_{k}^{\Ttran}Ax_{k}}{x_{k}^{\Ttran}Bx_{k}} = \frac{\lambda_{1}+t_{k}^{2}\lambda_{k}}{u_{1}^{\Ttran}Bu_{1}+2t_{k}u_{1}^{\Ttran}Bu_{k}+t_{k}^{2}u_{k}^{\Ttran}Bu_{k}}
		\end{equation*}
		is a rational function. We know that $\mu_{k}(t)$ is a constant, specifically, $\mu_{k}(t_{k})=\mu_{k}(0)=\mu$. Thus, we know that $Au_{k}=\mu Bu_{k}$ for all $k$. Then the proof is finished since $[u_{1},\dotsc,u_{n}]$ is a basis of $\R^{n}$.

	\subsection{Proof of \cref{estu1}}
	\label{appestu1}
	
		First, we claim that, for any $x$, there exist a scaled decomposition
		\begin{equation}
			\label{estu1x0}
			\alpha x= \widehat{u}_{1}\cos\theta+v\sin\theta,
		\end{equation}
		where $\widehat{u}_{1}$ is an $A$-unit eigenvector of $(A,M)$ associated with $\lambda_{1}$, $v$ is $A$-unit and $B$-orthogonal to  $\widehat{u}_{1}$, $\alpha$ is a scalar and $0\leq\theta\leq \pi/2$. Consider the $B$-orthogonal decomposition of $x$ as
		\begin{equation*}
			x = \widehat{u}_{1}(\widehat{u}_{1}^{\Ttran}Bx)+\Bigl(I-\frac{\widehat{u}_{1}\widehat{u}_{1}^{\Ttran}B}{\widehat{u}_{1}^{\Ttran}B\widehat{u}_{1}}\Bigr)x.
		\end{equation*}
		Let
		\begin{equation*}
			z =\Bigl(I-\frac{\widehat{u}_{1}\widehat{u}_{1}^{\Ttran}B}{\widehat{u}_{1}^{\Ttran}B\widehat{u}_{1}}\Bigr)x,
			\quad  v  = z/\norm{z}_{A},
			\quad  \theta = \arctan \Bigl(\frac{\norm{z}_{A}}{\abs{\widehat{u}_{1}^{\Ttran}Bx}}\Bigr)
			\quad\text{and}\quad  \alpha = \frac{\sin\theta}{\norm{z}_{A}}.
		\end{equation*}
		It is easy to verify that the decomposition \cref{estu1x0} holds. Since $\varsigma_{\min}$, $\varsigma_{\max}$ and the Rayleigh quotient $\Rq(x)$ are homogeneous of $x$, we can rescale $x$ in \cref{estu1x0} as
		\begin{equation}
			\label{estu1x}
			x= \widehat{u}_{1}\cos\theta+v\sin\theta.
		\end{equation}
		
		Now we can estimate $\varsigma_{\min}$ and $\varsigma_{\max}$. Due to the collinearity of $u_{1}$ and $\widehat{u}_{1}$, we know 
		\begin{equation}
			\label{decompsigma}
			\frac{x^{\Ttran}Ax}{x^{\Ttran}Bx} = \frac{x^{\Ttran}Ax}{\widehat{u}_{1}^{\Ttran}A\widehat{u}_{1}}\cdot \frac{\widehat{u}_{1}^{\Ttran}B\widehat{u}_{1}}{x^{\Ttran}Bx}\cdot\sigma.
		\end{equation}
		It is sufficient to bound the first two terms.
		By the decomposition \cref{estu1x}, we have
		\begin{equation*}
			\frac{x^{\Ttran}Ax}{\widehat{u}_{1}^{\Ttran}A\widehat{u}_{1}} = 1+2\cos\theta\sin\theta\cdot \widehat{u}_{1}^{\Ttran}Av.
		\end{equation*}
		Note that $\abs{\widehat{u}_{1}^{\Ttran}Av}\leq \cos\vartheta$. We obtain bounds for the first term in \cref{decompsigma} as 
		\begin{equation}
			\label{estuAx}
			1-2\cos\theta\sin\theta\cos\vartheta
			\leq \frac{x^{\Ttran}Ax}{\widehat{u}_{1}^{\Ttran}A\widehat{u}_{1}}
			\leq 1+2\cos\theta\sin\theta\cos\vartheta.
		\end{equation}
		For the second term in \cref{decompsigma}, similarly, according to the decomposition \cref{estu1x}, we have
		\begin{equation*}
			\frac{\widehat{u}_{1}^{\Ttran}B\widehat{u}_{1}}{x^{\Ttran}Bx}
			=\frac{\widehat{u}_{1}^{\Ttran}B\widehat{u}_{1}}{\widehat{u}_{1}^{\Ttran}B\widehat{u}_{1}\cos^{2}\theta+v^{\Ttran}Bv\sin^{2}\theta}
			=\frac{\widehat{u}_{1}^{\Ttran}B\widehat{u}_{1}}{\widehat{u}_{1}^{\Ttran}B\widehat{u}_{1}+(v^{\Ttran}Bv-\widehat{u}_{1}^{\Ttran}B\widehat{u}_{1})\sin^{2}\theta}.
		\end{equation*}
		Since $\norm{\widehat{u}_{1}}_{A}=\norm{v}_{A}=1$, we know 
		\begin{equation*}
			\frac{v^{\Ttran}Bv}{\widehat{u}_{1}^{\Ttran}B\widehat{u}_{1}} =
			\frac{v^{\Ttran}Bv}{v^{\Ttran}Av}\cdot 
			\frac{\widehat{u}_{1}^{\Ttran}A\widehat{u}_{1}}{\widehat{u}_{1}^{\Ttran}B\widehat{u}_{1}} = \sigma \frac{v^{\Ttran}Bv}{v^{\Ttran}Av}. 
		\end{equation*}
		Combining these two equations above with \cref{defnu}, we know that 
		\begin{equation}
			\label{estuBx}
			\frac{1}{1+(\sigma\nu_{\min}^{-1}-1)\sin^{2}\theta}\leq \frac{\widehat{u}_{1}^{\Ttran}B\widehat{u}_{1}}{x^{\Ttran}Bx}\leq \frac{1}{1-(1-\sigma\nu_{\max}^{-1})\sin^{2}\theta}.
		\end{equation}
		Combining \cref{decompsigma,estuAx,estuBx}, we provide bounds for $\varsigma_{\min}$ and $\varsigma_{\max}$ as 
		\begin{equation}
			\label{estsigma}
			\frac{\sigma(1-2\cos\theta\sin\theta\cos\vartheta)}{1+(\sigma\nu_{\min}^{-1}-1)\sin^{2}\theta}\leq \varsigma_{\min}\leq   
			\varsigma_{\max}\leq \frac{\sigma(1+2\cos\theta\sin\theta\cos\vartheta)}{1-(1-\sigma\nu_{\max}^{-1})\sin^{2}\theta}. 
		\end{equation}
	
		The last step is to provide an upper bound of $\sin\theta$ for all $x$ satisfying $\Rq(x)\leq \rho$. Expanding $x^{\Ttran}Mx$ with \cref{estu1x} we obtain
		\begin{equation*}
			\frac{\cos^{2}\theta}{\lambda_{1}}+\frac{2\cos\theta\sin\theta\cos\vartheta}{\lambda_{1}}+\frac{\sin^{2}\theta}{\varrho}
			\geq x^{\Ttran}Mx
			\geq \frac{\norm{x}_{A}^{2}}{\rho}
			\geq \frac{1}{\rho}(1-2\cos\theta\sin\theta\cos\vartheta),
		\end{equation*}
		where the first and last inequalities comes from the definition of $\vartheta=\vartheta(B,\rho;\Rq)$ and $\varrho=\varrho(B,\rho;\Rq)$.
		Rearranging inequalities above with trigonometric formulas, we have
		\begin{equation*}
			\rho^{-1}-\frac{\lambda_{1}^{-1}+\varrho^{-1}}{2}-\frac{(\lambda_{1}^{-1}-\varrho^{-1})\cos2\theta}{2}
			\leq (\lambda_{1}^{-1}+\rho^{-1})\cos\vartheta\sin2\theta.
		\end{equation*}
		Using trigonometric formulas and the inequality above, we know
		\begin{equation*}
			\begin{aligned}
				(\lambda_{1}^{-1}-\lambda_{2}^{-1})\sin^{2}\theta=        & \rho^{-1}-\frac{\lambda_{1}^{-1}+\varrho^{-1}}{2}
				-\frac{(\lambda_{1}^{-1}-\varrho^{-1})\cos2\theta}{2}\\ 
				&+(\lambda_{1}^{-1}-\rho^{-1})+(\varrho^{-1}-\lambda_{2}^{-1})\sin^{2}\theta                                                                      \\
				\leq     & (\lambda_{1}^{-1}+\rho^{-1})\cos\vartheta\sin2\theta + (\lambda_{1}^{-1}-\rho^{-1})+(\varrho^{-1}-\lambda_{2}^{-1})\sin^{2}\theta.
			\end{aligned}
		\end{equation*}
		Rearranging the inequality above, we obtain a quadratic inequality of $\sin\theta$ as 
		\begin{equation*}
			\begin{aligned}
				(\lambda_{1}^{-1}-\varrho^{-1})\sin^{2}\theta-2(\lambda_{1}^{-1}+\rho^{-1})\cos\vartheta\cos\theta\sin\theta-(\lambda_{1}^{-1}-\rho^{-1})\leq 0.
			\end{aligned}
		\end{equation*}
		Solving the quadratic inequality we obtain 
		\begin{equation*}
			\begin{aligned}
				\sin\theta &\leq \frac{2(\lambda_{1}^{-1}+\rho^{-1})\cos\vartheta\cos\theta+\bigl((\lambda_{1}^{-1}-\varrho^{-1})(\lambda_{1}^{-1}-\rho^{-1})\bigr)^{1/2}}{\lambda_{1}^{-1}-\varrho^{-1}}\\ 
				&\leq  \frac{2(\lambda_{1}^{-1}+\rho^{-1})}{\lambda_{1}^{-1}-\varrho^{-1}}\cos\vartheta+\Bigl(\frac{\lambda_{1}^{-1}-\rho^{-1}}{\lambda_{1}^{-1}-\varrho^{-1}}\Bigr)^{1/2}\\ 
				&\leq \frac{2(\lambda_{1}^{-1}+\rho^{-1})\cos\vartheta}{(\lambda_{1}^{-1}-\lambda_{2}^{-1})\bigl(1-(\cos\vartheta+\sqrt{\varepsilon})^{2}\bigr)}
				+\Bigl(\frac{\varepsilon}{1-(\cos\vartheta+\sqrt{\varepsilon})^{2}}\Bigr)^{1/2}\\ 
				&\leq \frac{8\lambda_{2}\cos\vartheta}{\lambda_{2}-\lambda_{1}}+\sqrt{2\varepsilon}=\sqrt{\varepsilon_{*}},
			\end{aligned}
		\end{equation*}
		where we use \cref{defvarrho} in the third inequality, and $(\cos\vartheta+\sqrt{\varepsilon})^{2}\leq \varepsilon_{*}\leq 1/2$ in the last inequality. Substituting $\sin\theta\leq \sqrt{\varepsilon_{*}}$ into \cref{estsigma}, we finish the proof.
	\subsection{Proof of \cref{rmkPEVP}}
	\label{apprmkPEVP}
	Before approximate $\mu_{B}$ and $L_{B}$. We compute the approximation of $\varsigma_{\min}$ and $\varsigma_{\max}$ first. By \cref{estu1}, we know that 
	\begin{equation*}
		\begin{aligned}
			\varsigma_{\min}/\sigma &\geq 1-\order\bigl(\sqrt{\varepsilon_{*}}\cos\vartheta+(\sigma\nu_{\min}^{-1}-1)\varepsilon_{*}\bigr) 
			\geq 1-\order\bigl(\sigma\nu_{\min}^{-1}(g^{-2}\cos^{2}\vartheta+\varepsilon)\bigr),\\ 
			\varsigma_{\max}/\sigma &\leq 1+\order\bigl(\sqrt{\varepsilon_{*}}\cos\vartheta+(1-\sigma\nu_{\max}^{-1})\varepsilon_{*}\bigr)
			\leq 1+\order(g^{-2}\cos^{2}\vartheta+\varepsilon).
		\end{aligned}
	\end{equation*}
	Now, let us turn to $\mu_{B}$:
	\begin{equation*}
			\mu_{B} = \frac{2\nu_{\min}\lambda_{1}}{\sigma}\Bigl(1-\frac{\lambda_{1}}{\lambda_{2}}\Bigr)
			\Bigl(\frac{\sigma(1-\frac{\rho}{\varrho})}{\varsigma_{\max}(1-\frac{\lambda_{1}}{\lambda_{2}})}-
			\frac{\sigma C_{\mathcal{X}}}{2\nu_{\min}\lambda_{1}g}\Bigr),	
	\end{equation*}
	where 
	\begin{equation*}
		C_{\mathcal{X}}  = \frac{8\nu_{\max}\rho}{\varsigma_{\min}}
				\biggl(
				\frac{\rho-\lambda_{1}}{\lambda_{1}}
				+\Bigl(\frac{\rho-\lambda_{1}}{\lambda_{1}}\Bigr)^{1/2}\cos\vartheta
				\biggr).
	\end{equation*}
	For the term without $C_{\mathcal{X}}$, by \cref{defvarrho}, we have 
	\begin{equation*}
		\begin{aligned}
			\frac{\sigma(1-\frac{\rho}{\varrho})}{\varsigma_{\max}(1-\frac{\lambda_{1}}{\lambda_{2}})} 
		&= \frac{\sigma}{\varsigma_{\max}}
		\Bigl(1-\frac{(\rho-\lambda_{1})\lambda_{2}+(\lambda_{2}-\varrho)\lambda_{1}}{\varrho(\lambda_{2}-\lambda_{1})}\Bigr)
		\geq 1-\order(g^{-2}\cos^{2}\vartheta+\varepsilon).
		\end{aligned}
	\end{equation*}
	For the term with $C_{\mathcal{X}}$, we have 
	\begin{equation*}
		\frac{\sigma C_{\mathcal{X}}}{2\nu_{\min}\lambda_{1}g} = \frac{8\sigma\kappa_{\nu}\rho }{2\varsigma_{\min}\lambda_{1}g}\biggl(
			\frac{\rho-\lambda_{1}}{\lambda_{1}}
			+\Bigl(\frac{\rho-\lambda_{1}}{\lambda_{1}}\Bigr)^{1/2}\cos\vartheta
			\biggr)\leq \order\bigl(\kappa_{\nu}(\varepsilon+\sqrt{\varepsilon/g}\cos\vartheta)\bigr).
	\end{equation*}
	Combining these two inequalities above, we have 
	\begin{equation*}
		\mu_{B} \geq  \frac{2\nu_{\min}\lambda_{1}}{\sigma}\Bigl(1-\frac{\lambda_{1}}{\lambda_{2}}\Bigr)\Bigl(1-\order\bigl(\kappa_{\nu}(g^{-2}\cos^{2}\vartheta+\varepsilon)\bigr)\Bigr).
	\end{equation*}
	
	For the term $L_{B}$:
	\begin{equation*}
		L_{B} = \frac{2\nu_{\max}\lambda_{1}}{\sigma}\Bigl(1-\frac{\lambda_{1}}{\lambda_{n}}\Bigr)\Bigl(\frac{\rho\sigma}{\lambda_{1}\varsigma_{\min}}+\frac{\sigma\lambda_{n} C_{\mathcal{X}}}{2\nu_{\max}\lambda_{1}(\lambda_{n}-\lambda_{1})}\Bigr),
	\end{equation*}
	with almost same operations, we know 
	\begin{equation*}
		L_{B} \leq  \frac{2\nu_{\max}\lambda_{1}}{\sigma}\Bigl(1-\frac{\lambda_{1}}{\lambda_{n}}\Bigr)\Bigl(1+\order\bigl(\sigma\nu_{\min}^{-1}(g^{-2}\cos^{2}\vartheta+\varepsilon)\bigr)\Bigr).
	\end{equation*}
	
	Thus, the condition number after preconditioning is 
	\begin{equation*}
		\kappa_{B} = \frac{L_{B}}{\mu_{B}} \leq \kappa_{\nu}\frac{1-\lambda_{1}/\lambda_{n}}{1-\lambda_{1}/\lambda_{2}}\Bigl(1-\order\bigl(\kappa_{\nu}(g^{-2}\cos^{2}\vartheta+\varepsilon)\bigr)\Bigr).
	\end{equation*}
	
	\subsection{Proof of \cref{propestSeq}}
	\label{apppropestSeq}
	The proof \cref{propestSeq} highly relies on \cite{Zhang2018}.
	First, we show $(\alpha,\phi_{i}(x))$ is a weak estimate sequence, which is based on \cite[Lemma~4]{Zhang2018}. Note that in \cite[Lemma~4]{Zhang2018}, the only requirement for $\{y_{i}\}$ is that $\{y_{i}\}\subset \mathcal{X}$, which is satisfied in \cref{propestSeq}. Thus, by taking all $\alpha_{i}$ in \cite[Lemma~4]{Zhang2018} as $\alpha$ and $\gamma_{0}=\gamma$, we know that $(\alpha,\phi_{i}(x))$ is a weak estimate sequence since we use the same update formula for $v_{i}$ as \cite[Algorithm~1]{Zhang2018}, which depends only on $y_{i}$. 
	
	Then we show that $f(x_{i+1})\leq \phi_{i+1}^{*}$, which relies on \cite[Lemma~6]{Zhang2018}. In the proof of \cite[Lemma~6]{Zhang2018}, they show that 
	\begin{equation*}
		f(y_{i})-\frac{h}{2}\norm{\grad f(y_{i})}^{2}\leq \phi_{i+1}^{*}.
	\end{equation*}
	Thus, our choice of $x_{i+1}$ in \cref{algo} satisfies
	\begin{equation*}
		f(x_{i+1})\leq f(y_{i})-\frac{h}{2}\norm{\grad f(y_{i})}^{2}\leq \phi_{i+1}^{*}.
	\end{equation*}
	
	Next, we show that $\phi_{i+1}(x_{*})-f(x_{*}) \leq (1-\alpha) \bigl(\phi_{i}(x_{*})-f(x_{*})\bigr)$. This uses the fact proved in \cite[Lemma~4]{Zhang2018} as
	\begin{equation*}
		\phi_{i+1}(x_{*})\leq \overline{\phi}_{i+1}(x_{*})=(1-\alpha)\phi_{i}(x_{*})+\alpha \bigl(f(y_{i})+\dual{\grad f(y_{i}),\log_{y_{i}}(x_{*})}+\frac{\mu}{2}\norm{\log_{y_{i}}(x_{*})}^{2}\bigr).
	\end{equation*}
	By the geodesic convexity of $f$, we know that 
	\begin{equation*}
		\phi_{i+1}(x_{*})\leq (1-\alpha)\phi_{i}(x_{*})+\alpha f(x_{*}) = (1-\alpha)\bigl(\phi_{i}(x_{*})-f(x_{*})\bigr)+f(x_{*}),
	\end{equation*}
	which is the desired result.
	
	At last, we show the convergence result \cref{mainprop}. By last two items we proved above, we know that 
	\begin{equation*}
		f(x_{i})-f(x_{*})\leq \phi_{i}^{*}-f(x_{*}) \leq \phi_{i}(x_{*})-f(x_{*})\leq (1-\alpha)^{i}\bigl(\phi_{0}(x_{*})-f(x_{*})\bigr).
	\end{equation*}
	The \cref{mainprop} is proved by the definition of $\phi_{0}$.
	\subsection{Proof of \cref{thmLORAG}}
	\label{appthmLORAG}
	
	Before proving \cref{thmLORAG}, we first state a distortion result for the Riemannian manifolds with bounded sectional curvature. 
	\begin{proposition}[{\cite[Theorem~10]{Zhang2018}}]
		\label{propdistort}
		For any three points $x$, $y$ and $z\in \mathcal{X}$, where $\mathcal{X}$ is a geodesically convex set with sectional curvature bounded in $[-K,K]$. Let
		\begin{equation*}
			\eta = \sqrt{1+2K\norm{\log_{x}(y)}^{2}},
		\end{equation*}
		then we have
		\begin{equation*}
			\begin{aligned}
				\norm{\log_{y}(x)-\log_{y}(z)}
				\leq \eta\norm{\log_{x}(z)}
				\leq \eta^{2} \norm{\log_{y}(x)-\log_{y}(z)}.
			\end{aligned}
		\end{equation*}
	\end{proposition}
	
	Combining \cref{propdistort} with the assumption $(1+KR_{1}^{2})^{2}\leq 1+\beta$ in \cref{thmLORAG}, we know that \cref{distortion} can be proved by $y_{i},\,y_{i+1}\in\mathcal{B}_{R_{1}}$. Moreover, we also obtain a triangular-type inequality as 
	\begin{equation}
		\label{conDistort}
		\norm{\log_{x_{*}}(z)}
		\leq \sqrt{1+\beta} \norm{\log_{y}(x_{*})-\log_{y}(z)},
	\end{equation}
	for any $y\in\mathcal{B}_{R_{1}}$ and $z\in\mathcal{M}$. 
	
	Now we can go back to the proof.
		By \cref{lemxk}, we know $f(x_{m})\leq f(x_{0})$ for all $m\geq 0$, then $x_{m}\in\mathcal{B}_{R_{1}}$ comes directly from \cref{invarset}. Once $v_{m}\in\mathcal{B}_{R_{1}}$ is proved, we obtain $y_{m}\in\mathcal{B}_{R_{1}}$ directly since $y_{m}$ is in the geodesic from $x_{m}$ to $v_{m}$ and $\mathcal{B}_{R_{1}}$ is geodesically convex.
	
		Now, we only need to prove $v_{m}\in\mathcal{B}_{R_{1}}$ for all $m\geq 0$. The proof is based on a mathematical induction. The case $m=0$ is straight from $v_{0}=x_{0}\in\mathcal{B}_{R_{1}}$. Assume $v_{i}\in\mathcal{B}_{R_{1}}$ for all $0\leq i\leq m$.
	
		First, we need to prove
		\begin{equation}
			\label{suffDecrease}
			f(x_{i+1})\leq f(y_{i})-\frac{1}{2L}\norm{\grad f(y_{i})}^{2}
		\end{equation}
		holds for all $0\leq i\leq m$. Since $y_{i}\in\mathcal{B}_{R_{1}}\subset \mathcal{X}$ and $\nabla f(x_{*})=0$, we know that
		\begin{equation}
			\label{normGrady}
			\norm{\grad f(y_{i})} = \norm{\grad f(y_{i})-\Gamma_{x_{*}}^{y_{i}}\grad f(x_{*})} \leq L\norm{\log_{x_{*}}(y_{i})}\leq LR_{1},
		\end{equation}
		where $\Gamma_{x_{*}}^{y_{i}}$ is the parallel transport. Let 
		\begin{equation*}
			\overline{y}_{i} = \exp_{y_{i}}\Bigl(-\frac{1}{L}\grad f(y_{i})\Bigr).
		\end{equation*}
		By the triangular inequality, we know
		\begin{equation*}
			\norm{\log_{x_{*}}(\overline{y}_{i})} = \norm{\log_{x_{*}}(y_{i})+\Gamma_{y_{i}}^{x_{*}}\log_{y_{i}}(\overline{y}_{i})}\leq \norm{\log_{x_{*}}(y_{i})}+\frac{1}{L}\norm{\grad f(y_{i})}\leq 2R_{1}\leq R_{2},
		\end{equation*}
		which implies $\overline{y}_{i}\in\mathcal{B}_{R_{2}}\subset \mathcal{X}$.
		Using the $L$-Lipschitz smoothness of $f$ in $\mathcal{X}$, we obtain 
		\begin{equation*}
			f(\overline{y}_{i})\leq f(y_{i})-\frac{1}{2L}\norm{\grad f(y_{i})}^{2}.
		\end{equation*}
		Then \cref{suffDecrease} is proved by the locally optimal step in \cref{algo}, \ie $f(x_{i+1})\leq f(\overline{y}_{i})$.
	
		Then we will show
		\begin{equation}
			\label{vinBR2}
			v_{m+1}\in\mathcal{B}_{R_{2}}\subset \mathcal{X}.
		\end{equation}
		From the recurrence of $v_{m+1}$ in \cref{algo} and the triangular inequality, we know
		\begin{equation*}
			\norm{\log_{x_{*}}(v_{m+1})} = \norm{\log_{x_{*}}(\overline{v}_{m})+\Gamma_{\overline{v}_{m}}^{x_{*}}\log_{\overline{v}_{m}}^{v_{m+1}}}\leq \norm{\log_{x_{*}}(\overline{v}_{m})}+\frac{\alpha}{\overline{\gamma}}\norm{\grad f(y_{m})},
		\end{equation*}
		where
		\begin{equation*}
			\overline{v}_{m} = \exp_{y_{m}}\Bigl(\frac{1-\alpha}{1+\beta}\log_{y_{m}}(v_{m})\Bigr)
		\end{equation*}
		is on the geodesic from $y_{m}$ to $v_{m}$. Note that $y_{m}$ and $v_{m}$ are both in $\mathcal{B}_{R_{1}}$, a geodesically convex set, we know that $\overline{v}_{m}$ is also in $\mathcal{B}_{R_{1}}$. Combining it with $\norm{\grad f(y_{m})}\leq LR_{1}$ in \cref{normGrady} and the triangular inequality, we prove \cref{vinBR2} by
		\begin{equation*}
			\norm{\log_{x_{*}}(v_{m+1})} \leq R_{1}+\frac{\alpha L}{\overline{\gamma}}R_{1} = \Bigl(1+\frac{\alpha L}{(1+\beta)\gamma}\Bigr)R_{1}\leq R_{2}.
		\end{equation*}
	
		The last step is to show that 
		\begin{equation}
			\label{vinBR1}
			v_{m+1}\in\mathcal{B}_{R_{1}}.
		\end{equation}
		We need to use the convergence established in \cref{propestSeq}. First, by \cref{suffDecrease}, we know \cref{algo} is a special case of \cref{GRAGD}. Then, we have shown that $x_{m+1}\in\mathcal{B}_{R_{1}}$, $v_{m+1}\in\mathcal{B}_{2}$, $y_{m}\in\mathcal{B}_{R_{1}}$ and $y_{m+1}\in\mathcal{B}_{R_{2}}$, which indicates that all iterates are in $\mathcal{X}$. At last, the distortion \cref{distortion} is handled by \cref{propdistort}. Using the convergence result in \cref{propestSeq}, we have 
		\begin{equation*}
			\phi_{m+1}(x_{*})-f(x_{*})\leq (1-\alpha)^{m+1}E_{0}\leq E_{0}.
		\end{equation*}
		Note that $\phi_{m+1}^{*}\geq f(x_{m+1})\geq f(x_{*})$ and
		\begin{equation*}
			\phi_{m+1}(x_{*})-f(x_{*}) = \phi_{m+1}^{*}-f(x_{*})+\frac{\overline{\gamma}}{2}\norm{\log_{y_{m}}(x_{*})-\log_{y_{m}}(v_{m+1})}^{2}.
		\end{equation*}
		We can prove \cref{vinBR1} by 
		\begin{equation*}
			\begin{aligned}
				\norm{\log_{x_{*}}(v_{m+1})}^{2}&\leq (1+\beta)\norm{\log_{y_{m}}(x_{*})-\log_{y_{m}}(v_{m+1})}^{2} \\ 
				&\leq \frac{2(1+\beta)}{\overline{\gamma}}\bigl(\phi_{m+1}(x_{*})-f(x_{*})\bigr) 
				\leq  \frac{2E_{0}}{\gamma}\leq R_{1}^{2},
			\end{aligned}
		\end{equation*}
		where \cref{conDistort} is used in the first inequality.
	\subsection{Proof of \cref{rmkcoef}}
	\label{apprmkcoef}
	It is clear that when $\kappa\geq 9$, all parameter, \ie $\alpha$, $\beta$ and $\gamma$, are positive. Note that 
	\begin{equation*}
		\frac{\alpha L}{(1+\beta)\gamma} = \frac{\dfrac{L}{2\sqrt{\kappa}}}{\dfrac{2\sqrt{\kappa}-1}{2\sqrt{\kappa}-4}\cdot\dfrac{(\sqrt{\kappa}-2)\mu}{2(2\sqrt{\kappa}-1)}} = 2\sqrt{\kappa}.
	\end{equation*}
	We know that 
	\begin{equation*}
		R_{2} = \Bigl(1+\frac{\alpha L}{(1+\beta)\gamma}\Bigr)R_{1} = (1+2\sqrt{\kappa})R_{1}\geq 2R_{1},
	\end{equation*}
	which is consistent with the condition in \cref{thmLORAG}.
	
	Now let us turn to upper bounds for $R_{1}$ and $R_{2}$. For $R_{1}^{2}$, by the geodesic convexity of $f$, we have
	\begin{equation*}
		E_{0} = f(x_{0})-f(x_{*})+\frac{\gamma}{2}\norm{x_{0}-x_{*}}^{2}\leq \frac{\mu+\gamma}{\mu}\bigl(f(x_{0})-f(x_{*})\bigr).
	\end{equation*}
	By $0<\alpha\leq 1/6$, we know $4< \mu/\gamma\leq 10$ and
	\begin{equation*}
		R_{1}^{2}\leq \frac{2E_{0}}{\gamma}= \frac{2(\mu+\gamma)}{\gamma\mu}\bigl(f(x_{0})-f(x_{*})\bigr) \leq \frac{22}{\mu}\bigl(f(x_{0})-f(x_{*})\bigr).
	\end{equation*}
	The estimation for $R_{2}$ comes from $R_{2}=(1+2\sqrt{\kappa})R_{1}$.
	
	\subsection{Proof of \cref{thmPLORAG}}
	\label{appthmPLORAG}
	It is sufficient to verify conditions in \cref{rmkcoef}. Using the $L_{B}$-Lipschitz smooth of $f$, we know that $\mathcal{B}_{R_{2}}\subset \mathcal{X}$ holds when
		\begin{equation*}
			\frac{L_{B}R_{2}^{2}}{2}\leq \rho-\lambda_{1}.
		\end{equation*}
		According to \cref{rmkcoef}, it is enough to show
		\begin{align}
			\label{thmR2}
			f(x_{0})-f(x_{*}) & \leq \frac{\mu_{B}}{22}\bigl(\sqrt{1+\beta}-1\bigr),                               \\
			\label{thmR3}
			f(x_{0})-f(x_{*}) & \leq \frac{(\rho-\lambda_{1})}{11(1+2\sqrt{\kappa_{B}})^{2}\kappa_{B}}.
		\end{align}
		Note that \cref{thmR3} is same as the condition \cref{conX0} since $f(x_{*})=\lambda_{1}$. We need to show that \cref{conX0} implies \cref{thmR2}, \ie 
		\begin{equation*}
			\frac{(\rho-\lambda_{1})}{11(1+2\sqrt{\kappa_{B}})^{2}\kappa_{B}}\leq \frac{\mu_{B}}{22}\bigl(\sqrt{1+\beta}-1\bigr).           
		\end{equation*}
		Since $0<\beta\leq 3/2$, we obtain that
		\begin{equation*}
			\sqrt{1+\beta}-1\geq \frac{\beta}{3} = \frac{1}{2\sqrt{\kappa_{B}}-4}\geq \frac{1}{2\sqrt{\kappa_{B}}}.
		\end{equation*}
		Combining it with $\kappa_{B}=L_{B}/\mu_{B}$, it is sufficient to show
		\begin{equation*}
			\frac{\mu_{B}(\rho-\lambda_{1})}{(1+2\sqrt{\kappa_{B}})^{2}L_{B}} \leq \frac{\mu_{B}}{4\sqrt{\kappa_{B}}}.
		\end{equation*}
		According to \cref{defLmu}, we have
		\begin{equation*}
			L_{B} = \frac{2\nu_{\max}\rho}{\varsigma_{\min}}\Bigl(1-\frac{\lambda_{1}}{\lambda_{n}}\Bigr)+C_{\mathcal{X}}
			\geq 2\rho\Bigl(1-\frac{\lambda_{1}}{\rho}\Bigr)
			\geq 2(\rho-\lambda_{1}).
		\end{equation*}
		Then the theorem is proved by $\kappa_{B}\geq 9$.
	
		\subsection{Proof of \cref{lemx0}}
		\label{applemx0}
		By the estimation of eigenvector, we know
		\begin{equation*}
			\Rq(\widetilde{x}_{0})-\lambda_{1}=\lambda_{0}-\lambda_{1}\lesssim \lambda_{1}H^{2}.
		\end{equation*}
		According to $x_{0}=B^{-1}\widetilde{x}_{0}$, we have
		\begin{equation*}
			x_{0}=\lambda_{0}^{-1}\widetilde{x}_{0}+B_{L}^{-1}\widetilde{x}_{0},
		\end{equation*}
		where $B_{L}$ is defined in \cref{estLoc}. Then, by \cref{estLoc},
		\begin{equation}
			\label{estx0hat1}
			\rho_{0}=\Rq(x_{0})
			=\frac{\norm{\lambda_{0}^{-1}\widetilde{x}_{0}+B_{L}^{-1}\widetilde{x}_{0}}_{A}^{2}}
			{\norm{\lambda_{0}^{-1}\widetilde{x}_{0}+B_{L}^{-1}\widetilde{x}_{0}}^{2}}
			= \frac{\lambda_{0}+\lambda_{0}\dual{\widetilde{x}_{0},B_{L}^{-1}\widetilde{x}_{0}}_{A}+\lambda_{0}^{2}\norm{B_{L}^{-1}\widetilde{x_{0}}}_{A}^{2}}
			{1+\lambda_{0}\dual{\widetilde{x}_{0},B_{L}^{-1}\widetilde{x}_{0}}+\lambda_{0}^{2}\norm{B_{L}^{-1}\widetilde{x}_{0}}^{2}}.
		\end{equation}
		Due to the local stability in \cref{aspAS}, we know
		\begin{equation*}
			\dual{\widetilde{x}_{0},B_{L}^{-1}\widetilde{x}_{0}}=\sum_{i=1}^{N}\dual{R_{i}\widetilde{x}_{0},B_{i}^{-1}R_{i}\widetilde{x}_{0}}\geq \frac{1}{\omega}\sum_{i=1}^{N}\norm{R_{i}^{\Ttran}B_{i}^{-1}R_{i}\widetilde{x}_{0}}_{A}^{2}\geq0.
		\end{equation*}
		Thus, the denominator in \cref{estx0hat1} is greater than or equal to $1$. For the numerator, we need to show that 
		\begin{equation}
			\label{estx0hat2}
			\dual{\widetilde{x}_{0},B_{L}^{-1}\widetilde{x}_{0}}_{A}\lesssim H^{2}\quad\text{and}\quad \norm{B_{L}^{-1}\widetilde{x_{0}}}_{A}\lesssim \lambda_{0}^{-1/2}H.
		\end{equation}
		The second inequality has already been proved in \cref{estLoc}. For the first one, let 
		\begin{equation*}
			r_{0} = \alpha_{0} \widetilde{x}_{0}-u_{1} = \alpha_{0} R_{0}^{\Ttran}u_{0}-u_{1},
		\end{equation*}
		where
		\begin{equation*}
			\alpha_{0} = \argmin_{\alpha\in\R}\norm{u_{1}-\alpha R_{0}^{\Ttran}u_{0}}_{A} = \frac{(R_{0}^{\Ttran}u_{0},u_{1})_{A}}{\lambda_{0}}\approx 1.
		\end{equation*}
		Using the Cauchy--Schwarz inequality, we have
		\begin{equation}
			\label{xbx}
			\begin{aligned}
				\alpha_{0}^{2} (\widetilde{x}_{0},B_{L}^{-1}\widetilde{x}_{0})_{A}
				 & =  (r_{0},B_{L}^{-1}r_{0})_{A}+2(u_{1},B_{L}^{-1}r_{0})_{A}+(u_{1},B_{L}^{-1}u_{1})_{A}                                   \\
				 & \leq \norm{r_{0}}_{A}\norm{B_{L}^{-1}r_{0}}_{A}+2\sqrt{\lambda_{1}}\norm{B_{L}^{-1}r_{0}}_{A}+(u_{1},B_{L}^{-1}u_{1})_{A}.
			\end{aligned}
		\end{equation}
		By the Poincar\'e inequality,  $C_{c}$ in \cref{propAS} and \cref{estLoc},  we know that 
		\begin{equation*}
			\norm{B_{L}^{-1}r_{0}}_{A}\lesssim \lambda_{1}^{-1/2}H \norm{r_{0}}\leq \lambda_{1}^{-1}H \norm{r_{0}}_{A} \lesssim \lambda_{1}^{-1/2}H^{2}.
		\end{equation*}
		Thus, the first two terms in \cref{xbx} are $\order(H^{2})$.
		For the other term, similar to \cref{estLoc1},
		\begin{equation*}
			(u_{1},B_{L}^{-1}u_{1})_{A} = \lambda_{1}(u_{1},B_{L}^{-1}u_{1}) = \lambda_{1}\sum_{i=1}^{N}(R_{i}u_{1},B_{i}^{-1}R_{i}u_{1})\lesssim H^{2}\sum_{i=1}^{N}\norm{R_{i}u_{1}}^{2}\lesssim H^{2},
		\end{equation*}
		where we use $C_{P}$ in \cref{propAS}. Thus, we obtain \cref{estx0hat2}. The lemma is proved by combining \cref{estx0hat1,estx0hat2}.

\end{appendices}

\end{document}